\documentclass[11pt]{amsart}
\usepackage{amsmath,amssymb,mathrsfs,color,amsthm}
\usepackage[title]{appendix}
\usepackage{amssymb}
\usepackage{graphicx}  
\usepackage{graphicx}  
\usepackage{caption}   
\usepackage{enumitem}
\usepackage{relsize}
\usepackage{graphicx}
\allowdisplaybreaks

\def\R{{\mathbb R}}

\newtheorem{thm}{Theorem}[section]
\newtheorem{cor}[thm]{Corollary}
\newtheorem{lem}[thm]{Lemma}
\newtheorem{prop}[thm]{Proposition}

\theoremstyle{definition}
\newtheorem{de}[thm]{Definition}
\theoremstyle{remark}
\newtheorem{rem}[thm]{Remark}

\newtheorem{ass}[thm]{\bfseries Assumption}
\numberwithin{equation}{section}

\newcommand{\rmd}{{\rm d}}
\newcommand{\rme}{{\rm e}}
\allowdisplaybreaks

\begin{document}

\title[A trajectorial approach to gradient flows for McKean-Vlasov SDEs]{A trajectorial approach to the gradient flow of McKean-Vlasov SDEs with mobility}

\author{Zhenxin Liu}
\address{Z. Liu: School of Mathematical Sciences, Dalian University of Technology, Dalian
116024, P. R. China}
\email{zxliu@dlut.edu.cn}

\author{Xuewei Wang}
\address{X. Wang (Corresponding author): School of Mathematical Sciences, Dalian University of Technology, Dalian
116024, P. R. China}
\email{Wangxueweii@163.com}

\date{January 20, 2025}



\begin{abstract}
We establish the gradient flow representation of diffusion with mobility $b$ with respect to the modified Wasserstein quasi-metric $W_h$, where $h(r)=rb(r)$.
The appropriate selection of the free energy functional depends on the specific form of the generalized entropy.
Different from the JKO scheme, we derive the trajectorial version of the relative entropy dissipation identity for the McKean-Vlasov stochastic differential equation (SDE) with Nemytskii-type coefficients, utilizing techniques from stochastic analysis.
Based on this, we demonstrate that the trajectorial average of the solution process to the McKean-Vlasov SDE, with respect to the underlying measure, corresponds to the rate of dissipation of the free energy.
As an application, we present the energy dissipation of the Fermi-Dirac-Fokker-Planck equation, a model widely used in physics and biology to describe saturation effects.
Inspired by numerical simulations, we propose two questions on condensation phenomena and non-exponential convergence rate.
\end{abstract}

\maketitle

\section{Introduction}
\setcounter{equation}{0}

We are interested in the following McKean-Vlasov SDE with coefficients of Nemytskii-type
\begin{equation}\label{MVSDE}
\begin{aligned}
\rmd X(t)&=-\nabla \Phi(X(t)) b\big(p(t,X(t))\big) \rmd t + \sqrt{\frac{2f\big(p(t,X(t))\big)}{p(t,X(t))}}\rmd W(t),  t \geq 0, \\
X(0)&=\xi_0
\end{aligned}
\end{equation}
on $\R^d$, where $W(t)$, $t\geq 0$, is a standard Brownian motion on a probability space $(\Omega, \mathcal{G}, \mathbb{P})$ with normal filtration $(\mathcal{G}(t))_{t\geq 0}$. The smooth function $\Phi: \R^d \rightarrow [0, +\infty)$ serves as a potential field and the smooth functions $b, f:\R \rightarrow [0, +\infty)$ represent the influence of the distribution on drift term and diffusion term, respectively.
Let $P(t):=Law(X(t))$ and $p(t,x)$ be the density of $P(t)$, $t \geq 0$.
Under appropriate assumptions, \eqref{MVSDE} has a pathwise unique strong solution and $p(t,x)$ satisfies the following nonlinear Fokker-Planck equation
\begin{equation}\label{NFPE}
\begin{aligned}
\partial_t p(t,x) &= {\rm div} \big(\nabla \Phi(x) b(p(t,x)) p(t,x)\big) + \Delta f(p(t,x)), (t, x)\in [0, \infty) \times \R^d, \\
p(0, x)&=p_0 (x), \quad x \in \R^d.
\end{aligned}
\end{equation}
In contrast to the linear Fokker-Planck equation, the nonlinear equation takes into account also the effects of particle mobility and nonlinear diffusion. The function $b(p)$ reflects the dependence of particle mobility on the density $p$.
This has significant applications in mathematical biology and mathematical physics, particularly in modeling phenomena that aim to prevent overcrowding \cite{CHR, NFPE}.
We say that $b$ is a saturation function if it is non-increasing and there exists a constant $\alpha>0$ such that $b(\alpha)=0$, and $(\alpha-s)b(s)>0$ for $s\neq \alpha$. The constant $\alpha$ is referred to as the saturation level \cite{BCH}. A typical example is $f(p)=p$ and $\alpha =1$ \cite{CLR}, in which case the equation is known as the Fermi-Dirac-Fokker-Planck equation.

The study of $W_2$-gradient flow for linear Fokker-Planck equations can be traced back to the work of Jordan, Kinderlehrer and Otto \cite{JKO2}. Notably, they introduced the $W_2$-metric and applied discretization methods to demonstrate that the heat equation is a $W_2$-gradient flow. This analytical framework, known as the JKO scheme, has been widely applied in the study of gradient flows for various equations \cite{CS, GM, LWL, MS}.
Very recently, Karatzas et al. \cite{KST} proposed a creative description of the $W_2$-gradient flow from the perspective of SDEs and employed stochastic analysis techniques to provide a trajectorial interpretation of the $W_2$-gradient flow for the classical SDE
\begin{equation}\label{langevin}
\begin{aligned}
\rmd X(t)&=-\nabla \Phi(X(t)) \rmd t + \sqrt2\rmd W(t), t\geq 0,\\
X(0)&=\xi_0.
\end{aligned}
\end{equation}
The equation \eqref{langevin} describes the motion of individual particles in the system, which is influenced solely by their own positions. However, the existence of inter-particle interactions implies that the motion of individual particles is also affected by all other particles in the system. To describe such interactions, Kac \cite{K} introduced McKean-Vlasov SDEs, which characterize these effects through the influence of the distribution on the drift and diffusion terms.
Due to their capability to more accurately model significant problems in fields such as physics and finance, McKean-Vlasov SDEs have attracted substantial attention. Naturally, the gradient flow of McKean-Vlasov SDEs has also become a topic of significant research interest.
In \cite{TC}, the trajectorial version of the $W_2$-gradient flow for the SDE with the nonlinear drift term $-\big( \nabla \Phi(X(t))+\nabla (W(x)\ast P(t))(X(t)) \big)$ is studied.
In \cite{KC}, the SDE without the drift term is analyzed on an open connected bounded domain, with normal reflection at the boundary.
In contrast to previous works, we focus on the case where both the drift and diffusion terms are dependent on the distribution $P$.
Under these conditions, the basic components for gradient flow of the McKean-Vlasov SDE, specifically the internal energy and the metric, will undergo modifications.
The quadratic Wasserstein metric $W_2$ and the Boltzmann-Gibbs-Shannon entropy $E(p)=\int_{\R^d} p \ln p \rmd x$ are no longer applicable in this context, as the equation \eqref{NFPE} is affected by the nonlinear terms $f$ and $b$. To conquer this difficulty, we find that the modified Wasserstein metric \cite{CLSS, DNS}
\begin{equation*}
W_h^2(\bar{\mu}_0, \bar{\mu}_1):=\inf \Big\{ \int_0^1 \!\! \int_{\R^d} |v_t|^2 h(u_t) \rmd x \rmd t \Big| (\mu, v) \in \mathcal{CE}(0,1), \rmd \mu_t=u_t \rmd x, \mu_i=\bar{\mu}_i, i=0,1 \Big\}
\end{equation*}
and the generalized entropy \cite{N2}
\begin{equation*}
E_g(p):=-\int_{\R^d} \eta(p) \rmd x
\end{equation*}
are unique choices which work in our case, where $h(p):=pb(p)$ and $\eta(r):=\int_0^r g(s)\rmd s := \int_0^r \int_1^s \frac{f'(w)}{wb(w)} \rmd w \rmd s $. For precise definitions, see Definitions \ref{modified} and \ref{relative.entropy}.
Correspondingly, the free energy functional $\mathcal{F}(p)$
is given by
\begin{equation*}\label{energy1}
\mathcal{F}(p(t,x)) :=\int_{\R^d} \eta (p(t,x)) +\Phi(x) p(t,x) \rmd x.
\end{equation*}

In the process of deriving the energy dissipation rate, a primary challenge lies in estimating the distribution. Since we consider the equation in the whole space and both the drift and diffusion terms are influenced by the distribution, estimating the behavior of the particle distribution at infinity becomes an unavoidable issue. We establish the results $P(t)\in \mathcal{P}_2(\R^d)$ and $\int_0^T \!\! \int_{\R^d} \frac{|\nabla p(t,x)|^2}{p(t,x)} \rmd x \rmd t<+\infty$, which will serve as a crucial foundation for subsequent analysis.
Subsequently, we derive the relative entropy dissipation
\begin{equation}\label{relative}
\dfrac{\rmd }{\rmd t}H_g(P(t)|P_{\infty})=-I_g(P(t)|P_{\infty})
:= \big\| \nabla\big(g(\rho(X(t_0)))+\Phi(X(t_0))\big) \big\|^2_{L^2(\hat{h}(P(t_0)))}
\end{equation}
and the Wasserstein metric derivative
\begin{equation*}
\begin{aligned}
\lim_{t \downarrow t_0} \dfrac{W_h\big(P(t), P(t_0)\big)}{t-t_0}
= \big( I_g(P(t_0)|P_{\infty}) \big)^{\frac{1}{2}},
\end{aligned}
\end{equation*}
where $P_{\infty}$ is the stationary distribution of \eqref{MVSDE}.
In fact, \eqref{relative} represents the probabilistic formulation of the energy dissipation
\begin{equation*}\label{general}
\dfrac{\rmd }{\rmd t}\mathcal{F}(p(t,x))= -\mathcal{I}(p(t,x)) :=\int_{\R^d} |\nabla\big(g(p(s,x))+\Phi(x)\big)|^2 b\big(p(s,x)\big)p(s,x) \rmd x.
\end{equation*}
For further details, refer to Subsection \ref{introduction}.
Since McKean-Vlasov SDEs allow for pathwise analysis, we can more precisely derive the trajectory version of the energy dissipation
\begin{equation*}
\lim_{t \downarrow t_0} \dfrac{\mathbb{E}_{\mathbb{P}}[\theta(t, X(t))| \mathcal{G}(t_0)] -\theta(t_0, X(t_0)) }{t-t_0}
= D(t_0, X(t_0)),
\end{equation*}
where $\theta(t,x):=\frac{\eta(p(t,x))}{p(t,x)}+\Phi(x)$ represents the trajectory energy
and the form of $D$ can be found in Theorem \ref{solution0}.
This implies that the energy dissipation of the system is entirely determined by the trajectories, and the process can be precisely described using SDEs.
Finally, based on the above analysis, we demonstrate that the solution to \eqref{NFPE} is the gradient flow of the free energy functional $\mathcal{F}$ with respect to the modified Wasserstein metric $W_h$, and the specific form of the $W_h$-gradient flow is given by
\begin{equation*}
{\rm  grad}_{W_h} \mathcal{F}(P(t))=-{\rm div}\Big( h\big(\rho(X(t))\big) \big(\nabla g ( \rho(X(t)) ) +\nabla \Phi(X(t)) \big) \Big).
\end{equation*}
When these results are applied to the linear case, where $f(p)=p$ and $b(p)=1$, we recover the well-known $W_2$-gradient flow.

The rest of this article is organized as follows. In Section \ref{pre}, we introduce the preliminaries necessary for the proof.
Section \ref{mainsection} presents the main results of the paper.
Subsection \ref{introduction} provides the definition and properties of the generalized entropy $E_g(p)$ and the energy dissipation identity.
The definition and properties of the modified Wasserstein metric $W_h$, along with the Wasserstein metric derivative, are given in Subsection \ref{wh}.
The $W_h$-gradient flow of \eqref{NFPE} is established in Subsection \ref{gf}.
In Section \ref{example}, we apply our theoretical results to the Fermi-Dirac-Fokker-Planck equation and provide numerical simulations.
Section \ref{discussion} presents several numerical results based on which two questions are raised, namely condensation phenomena and non-exponential convergence rate.
The proof of the square integrability of the logarithmic gradient $\frac{\nabla p(t,x)}{p(t,x)}$ from Section \ref{pre} is given in the Appendix.

\section{Preliminaries}\label{pre}

The terminal time is fixed at $T\in (0, \infty)$.
We denote by $\mathcal{P}(\R^d)$ the set of all probability measures on $\R^d$, and define the following subsets:
$$\mathcal{P}_2(\R^d):= \Big\{P\in \mathcal{P}(\R^d): \int_{\R^d}{|x|^2\rmd  P(x)<+\infty } \Big \}, $$
$$\mathcal{P}_{ac, 2}(\R^d):= \Big\{P\in \mathcal{P}_2(\R^d): P \text{ is a absolutely continuous measures with respect to } \rmd x \Big \}.$$
For convenience, we adopt the notation $\rho(X(t)):=p(t, X(t))$. For the smooth function $h: \R \rightarrow [0, \infty)$, we define
\begin{equation*}
\begin{aligned}
\hat{h}: \mathcal{P}_{ac, 2}(\R^d) &\rightarrow \mathcal{P}_{ac, 2}(\R^d) \\
\mu &\mapsto \hat{h}(\mu)
\end{aligned}
\end{equation*}
where $\rmd \mu=u \rmd x$ and $\rmd \hat{h}(\mu):= h(u)\rmd x$.

To ensure that the solutions of \eqref{MVSDE} and \eqref{NFPE} exhibit the regularity necessary for the subsequent proofs, we introduce the following assumptions.

\begin{ass}[Regularity conditions]\label{con} \quad
\begin{enumerate}[label=(\roman*)]
  \item \label{f} For the smooth function $f:\R\rightarrow [0, +\infty)$, $f(0)=0$, $\gamma_1 \leq f'(r) \leq \gamma_2 $ for all $r \in \R$, for $0 < \gamma_1 <\gamma_2 <+\infty$.
  \item \label{b} For the smooth function $b:\R\rightarrow [0, +\infty)$, $b_0 \leq b(s) \leq b_1$ for all $s\in \R$, for $0 < b_0 <b_1<+\infty$.
  \item \label{phi} For the smooth function $\Phi:\R^d\rightarrow [0, +\infty)$, there exist $C, R \geq 0$, such that $|\nabla \Phi(x)|\leq C|x| $ for all $|x|>R$.
  \item \label{initial} The probability distribution $P(0)\in \mathcal{P}_{ac, 2}(\R^d) $ and the corresponding probability density function $p_0 \in L^{\infty}(\R^d)$, $\mathcal{F}(p_0) <M_0$ for some constant $M_0\in \R$.
  \item \label{strong} \eqref{MVSDE} has a pathwise unique strong solution $(X(t))_{0\leq t\leq T}$.
  \item \label{smooth} \eqref{NFPE} has a unique smooth solution $p$ and $p$ is the probability density of $P$. Additionally, there exists a unique stationary solution $p_{\infty}$ such that
$$p(t,x) \in L^{\infty}([0,\infty) \times \R^d), \quad p(t,x)\geq 0 \text{ on } [0, \infty)\times \R^d, $$
$$\int_{\R^d}p(t,x)\rmd x =\int_{\R^d}p_0(x)\rmd x, \forall t \geq 0,\quad p_{\infty} =g^{-1}(-\Phi+ c) \text{  for some } c \in \R. $$
\end{enumerate}
\end{ass}

\begin{rem}
\cite[Thm 4.1, Thm 4.2]{Grube2023} and \cite[prop 2.2, Thm 4.1]{BR} provide a set of sufficient conditions for \ref{strong} and \ref{smooth} of Assumption \ref{con}. \\
\end{rem}

\begin{prop}\label{P2}
Suppose Assumption \ref{con} holds.
Then the pathwise unique strong solution of \eqref{MVSDE} satisfying $$P(t)\in \mathcal{P}_2(\R^d), t\in [0, T], $$
and the corresponding smooth solution $p$ of \eqref{NFPE} is bounded.
\end{prop}

\begin{proof}
We aim to prove that $P(t)\in \mathcal{P}_2(\R^d)$, which implies that $p(t, x)$ has a finite second moment for $t\in [0, T]$.
According to Assumption \ref{con} \ref{phi}, there exist $C , R\geq 0$ such that
\begin{equation}\label{x2}
\langle x, \nabla \Phi \rangle \geq -C|x|^2  \text{ for all } |x|>R.
\end{equation}
Let
$$m_R:=\max_{|x|\leq R}|\langle x, \nabla \Phi(x) \rangle | < \infty \text{ and } \tau_k:=\inf \{ t\geq 0, |X(t)|> k\} \text{ for integers } k>R. $$
From It\^{o} formula, we have
\begin{equation}\label{ito}
\begin{aligned}
\rmd |X(t)|^2
=& \Big( \dfrac{2d \cdot f\big(\rho(X(t))\big)}{\rho(X(t))}-2\Big \langle X(t), \nabla \Phi(X(t))b\big(\rho(X(t))\big) \Big \rangle \Big)\rmd t \\
 & +2X(t) \sqrt{\frac{2 f\big(\rho(X(t))\big)}{\rho(X(t))}} \rmd W(t).
\end{aligned}
\end{equation}
Denote $\phi_k(t):=\mathbb{E}_{\mathbb{P}}[|X(t \wedge \tau_k)|^2] $, $\phi (t):=\mathbb{E}_{\mathbb{P}}[|X(t)|^2]$, and we take expectation in \eqref{ito} to get
\begin{equation*}
\begin{aligned}
\phi_k(t)=&\phi(0) +\mathbb{E}_{\mathbb{P}}\Big\{ \int_0^{t \wedge\tau_k}\Big[\dfrac{2d \cdot f\big(\rho(X(t))\big)}{\rho(X(t))}-2\Big \langle X(t), \nabla \Phi(X(t))b\big( \rho(X(t)) \big) \Big \rangle \Big]\cdot \\
& \big[1_{X(t)> R}+1_{X(t)\leq R}\big] \rmd t \Big\} \\
\leq& \phi(0) +2d\gamma_1 t+2m_R b_1 t +2b_0 C \int_0^t\phi_k(u) \rmd u.
\end{aligned}
\end{equation*}
Applying the Gronwall inequality, we obtain
\begin{equation*}
\phi_k(t)\leq \phi(0)+(2m_R b_1+2d\gamma_1)t +2b_0 C\int_0^t[\phi(0)+2m_R b_1 u+2d\gamma_1 u]\rme^{2b_0 C(t-u)}\rmd u=:\varsigma(t).
\end{equation*}
Assumption \ref{con} \ref{smooth} implies $\varsigma(t)$ is finite for $t\in [0,T]$, and therefore, letting $k \rightarrow +\infty$, we have
\begin{equation*}
\phi(t)=\mathbb{E}_{\mathbb{P}} [|X(t)|^2] \leq \varsigma(t) <+\infty.
\end{equation*}
Hence $P(t)\in \mathcal{P}_2(\R^d)$ for $t\in [0,T]$.

Since $p$ is a smooth solution of \eqref{NFPE} and $p \in L^{\infty}([0,\infty) \times \R^d)$, we deduce the boundedness of $p$.
\end{proof}

Next, we establish the square integrability of the logarithmic gradient $\frac{\nabla p}{p}$ with respect to $P$ on $[0, T] \times \R^d$, which will be used in the proof of main results. We adopt the convention that $\frac{\nabla p(t,x)}{p(t,x)}:=0$ if $p(t,x)=0$.

\begin{lem}\label{integrability}
Suppose Assumption \ref{con} holds. Then $p(t, \cdot ) \in W_{\text{loc}}^{1,1}(\R^d)$ and
$$\int_0^T \!\! \int_{\R^d} \dfrac{|\nabla p(t,x)|^2}{p(t,x)} \rmd x \rmd t<+\infty. $$
\end{lem}

Since the proof of Lemma \ref{integrability} is lengthy and technical, it is deferred to the Appendix.

\section{Main result}\label{mainsection}

In this section, we will give the rate of energy dissipation and Wasserstein metric derivative, and establish the results for the $W_h$-gradient flow.

\subsection{Dissipation of free energy}\label{introduction}

We begin our analysis by introducing the definitions of the generalized relative entropy and the modified Fisher information, and then establish the relationship between these concepts, the free energy, and the energy dissipation functional.

For \eqref{langevin}, the internal energy is given by the Boltzmann-Gibbs-Shannon entropy
\begin{equation}\label{entropy}
E(p)=\int_{\R^d} p \ln p \rmd x.
\end{equation}
and the relative entropy(Kullback-Leibler divergence) has the form
\begin{equation*}\label{Kdiv}
H(P|P_{\infty})=\mathbb{E}_{\mathbb{P}}\Big[\ln \dfrac{P}{P_{\infty}}\Big],
\end{equation*}
where $p_{\infty}(x)=c \rme^{-\Phi(x)}$ is the unique stationary solution and $\rmd P_{\infty}= p_{\infty} \rmd x$.
It should be noted that
\begin{equation*}
H(P|P_{\infty})= E(p)-E(p_{\infty}).
\end{equation*}

It is pointed out in \cite{N2} that the form of \eqref{entropy} in thermodynamics arises from $\frac{1}{\mathbb{T}}=\frac{-p'(E)}{p(E)}$, where $\mathbb{T}$ represents the temperature.
Consequently, by replacing the denominator $p(E)$ with the nonlinear term $\psi(p(E))$, we capture the nonlinear effect of the probability density $p(E)$ on the energy.
Here, $\psi(s)$ is a positive function defined on $[0, \infty)$. To clarify the notation, we set $g_{\psi}(s):=\int_1^s \frac{1}{\psi(t)} \rmd t$, $G_{\psi}(s):=\int_1^s g_{\psi}(t) \rmd t$ and $\omega_{\psi}(x):=(x-1)G_{\psi}(0)-x G_{\psi}(\frac{1}{x})$.
We now proceed to formally present the definitions.

\begin{de}[Generalized entropy and generalized relative entropy]\label{relative.entropy}
If $\psi$ is a positive function defined on $[0, \infty)$ and $p$ and $q$ are probability densities, we define {\em the generalized entropy} by
\begin{equation*}
E_g(p):=\int_{\R^n} \omega_{\psi}\Big(\dfrac{1}{p}\Big)p \rmd x
\end{equation*}
and {\em the generalized relative entropy} of $P$ with respect to $Q$ by
\begin{equation*}
\begin{aligned}
H_g[P|Q]
&:=\int_{\R^n} G_{\psi}(p)-G_{\psi}(q)-G_{\psi}'(q)(p-q) \rmd x,
\end{aligned}
\end{equation*}
where $\rmd P =p\rmd x, \rmd Q=q\rmd x$.
\end{de}

\begin{rem}\label{prop.energy}
(i) Since $\psi$ is positive, we obtain that $g_{\psi}$ is strictly increasing and $G_{\psi}$ is convex. The convexity of $G_{\psi}$ implies that $E_g(p)$ is concave. Thus, $H_g(P|Q)$ is nonnegative, and $H_g(P|Q)=0$ holds only when $p=q$. \\
(ii) Unlike Boltzmann-Gibbs-Shannon entropy, which is additive, the generalized entropy is typically non-additive. \\
(iii) The definition of $g_{\psi}$ can be modified to $g_{\psi}(s)=\int_a^s \frac{1}{\psi(t)} \rmd t$ for an appropriate constant $a$. This adjustment avoids the issue where discontinuities in $\frac{1}{\psi}$ could lead to nonintegrability. If $a$ and $b$ are chosen appropriately, then $g_{\psi}(s)=\int_a^s \frac{1}{\psi(t)} \rmd t=\int_b^s \frac{1}{\psi(t)} \rmd t+C=\tilde{g}_{\psi}(s)+C$ and $E_g(p)=E_{\tilde{g}}+C$, where $C$ is a constant. In particular, $H_g(P|Q)=H_{\tilde{g}}(P|Q)$.
Since we focus on the dissipation rate of relative entropy $H_g(P|Q)$, this modification will not affect the subsequent results.
\end{rem}

By setting $\psi(s):=\frac{sb(s)}{f'(s)}$, we obtain $\int_{\R^n} \eta(p) \rmd x =-E_g(p)-G(0)$. Hence, we define the internal energy as $\int_{\R^n} \eta(p) \rmd x$ and the free energy as $\mathcal{F}(p) =\int_{\R^d} \eta (p) +\Phi p \rmd x$.
Let $p_{\infty}$ be the unique stationary solution of \eqref{NFPE} and $\rmd P_{\infty}=p_{\infty} \rmd x$. Then
\begin{equation*}
\mathcal{F}(p_{\infty})=\int_{\R^d}\eta(p_{\infty})+\Phi p_{\infty} \rmd x =-\int_{\R^d}\!\! \int_0^{p_{\infty}} \dfrac{f'(s)}{b(s)} \rmd s \rmd x + \int_{\R^d} c p_{\infty} \rmd x\geq c-\dfrac{\gamma_2}{b_0} > -\infty
\end{equation*}
and
\begin{equation*}
\begin{aligned}
H_g[P|P_{\infty}]
=\int_{\R^n} \eta(p) -\eta(p_{\infty}) +\Phi p-\Phi p_{\infty} \rmd x
=\mathcal{F}(p)-\mathcal{F}(p_{\infty}).
\end{aligned}
\end{equation*}
Combining Remark \ref{prop.energy} (i), it follows that $p_{\infty}$ is the unique minimizer of $\mathcal{F}(p)$ and $\mathcal{F}(p)\geq \mathcal{F}(p_{\infty})>-\infty$.

As mentioned in Introduction, the relative Fisher information is a quantity closely related to the relative entropy $H[P|Q]$.
For the linear Fokker-Planck equation \eqref{langevin}, the relative Fisher information of $P(t)$ with respect to $P_{\infty}$ is defined as $$I(P(t)|P_{\infty})=\mathbb{E}_{\mathbb{P}}\big[\big| \nabla \ln \big(\frac{P(t)}{P_{\infty}}\big) \big|^2\big] =\int_{\R^d}|\nabla \ln p-\nabla \ln p_{\infty}|^2 p\rmd x$$ and $I(P(t)|P_{\infty})$ serves as the energy dissipation functional.
Since the generalized relative entropy $H_g[P|P_{\infty}]$ is adopted here, the form of the relative Fisher information undergoes a modification.
We replace $\ln p$ by $\eta'(p)$, and the corresponding score function is defined as $\nabla \eta'(p)$, see \cite{B, CT} for details in the linear case.
Taking into account the influence of mobility $b$, we provide the following definition.

\begin{de}[Modified Fisher information and relative modified Fisher information]\label{Fisherinformation}
If $P$, $Q$ are probability measures and function $h:\R \rightarrow \R$, {\em the modified Fisher information} is defined by
\begin{equation*}
I_g(P):=\mathbb{E}_{\hat{h}(\mathbb{P})}\Big[\big| \nabla \eta'(P) \big|^2\Big],
\end{equation*}
and {\em the relative modified Fisher information} of $P$ with respect to $Q$ is defined by
\begin{equation*}
I_g(P|Q):=\mathbb{E}_{\hat{h}(\mathbb{P})}\Big[\big| \nabla \big(\eta'(P)-\eta'(Q)\big) \big|^2\Big].
\end{equation*}
\end{de}

From the expression of the modified Fisher information,  it is evident that when the absolute value of the gradient of the likelihood function $\eta'(p)$ is smaller, the value of $I_g(P)$ will also be smaller.
Therefore, $I_g(P)$ serves as a measure of uncertainty or information about the variable $x$.
As the probabilities of events tend to average out, the amount of information available decreases, leading to a lower value of $I_g(P)$.
Conversely, when $I_g(P)$ is high, more information is available, and our estimate of $x$ becomes more accurate.

Let $\phi(u):=\frac{\eta(u)}{u}$, then $\theta(t,x)=\phi(p(t,x))+\Phi(x)$ and the free energy functional is given by
\begin{equation*}\label{energy}
\begin{aligned}
\mathcal{F}(p(t,x)) =\int_{\R^d} \theta(t,x) p(t,x) \rmd x =\mathbb{E}_{\mathbb{P}}[\theta(t, X(t))].
\end{aligned}
\end{equation*}
We now proceed to derive the energy dissipation identity.
Since Theorem \ref{solution0}, Corollaries \ref{cor2} and \ref{cor1} are special cases with perturbed results ($\beta =0$), we omit the proof here.

\begin{thm}\label{solution0}
Suppose Assumption \ref{con} holds.
Then the energy process $\theta(t,X(t))$ admits the semimartingale decomposition
\begin{equation*}\label{decomposition0}
\theta(t, X(t)) -\theta(0, X(0))= M(t) +F(t), 0\leq t\leq T.
\end{equation*}
Here $M(t)$ is the $L^2({\mathbb{P}})$-bounded martingale
\begin{equation*}\label{Mt0}
M(t)=\int_0^t \sqrt{\frac{2f\big(\rho(X(s))\big)}{\rho(X(s))}} \nabla \theta(s, X(s)) \rmd W(s)
\end{equation*}
and
\begin{equation*}\label{Ft0}
\begin{aligned}
 F(t) = &\int_0^t D(s, X(s)) \rmd s \\
 :=&\int_0^t \Big[\phi'\big(\rho(X(s))\big)\cdot {\rm div} \Big(\nabla\Phi(X(s)) b\big(\rho(X(s))\big)\rho(X(s))+\nabla f\big(\rho(X(s))\big)\Big) \\
 &+\Delta \theta(s, X(s)) \dfrac{f\big(\rho(X(s))\big)}{\rho(X(s))}- \Big \langle \nabla \theta(s, X(s)), \nabla\Phi(X(s)) b\big(\rho(X(s))\big) \Big \rangle\Big] \rmd s
\end{aligned}
\end{equation*}
satisfies
\begin{equation*}\label{Ft20}
\begin{aligned}
 \mathbb{E}_{\mathbb{P}} [F(t)]
 =-\int_0^t \mathcal{I}(p(s,x))\rmd s
>-\infty.
\end{aligned}
\end{equation*}
\end{thm}

According to the trajectorial results in Theorem \ref{solution0}, we can express the trajectorial rate of relative entropy dissipation using conditional expectation.

\begin{cor}\label{cor2}
Suppose Assumption \ref{con} holds. For all $t_0 \in [0,T)$, the trajectorial rate of relative entropy dissipation identity is given by
\begin{equation}\label{560}
\begin{aligned}
\lim_{t \downarrow t_0} \dfrac{\mathbb{E}_{\mathbb{P}}[\theta(t, X(t))| \mathcal{G}(t_0)] -\theta(t_0, X(t_0)) }{t-t_0}
= D(t_0, X(t_0)),
\end{aligned}
\end{equation}
where the limit exists in $L^1(\mathbb{P})$.
\end{cor}

By averaging the trajectorial results, we obtain the following energy dissipation identity.

\begin{cor}\label{cor1}
Suppose Assumption \ref{con} holds. Then the relative entropy dissipation is given by
\begin{equation}\label{cor3}
\dfrac{\rmd }{\rmd t}H_g(P(t)|P_{\infty})=-I_g(P(t)|P_{\infty}).
\end{equation}
\end{cor}

\begin{rem}
(i) Note that \eqref{cor3} corresponds to the energy dissipation
\begin{equation}\label{edi20}
\begin{aligned}
 \dfrac{\rmd }{\rmd t}\mathcal{F}(p(t,x))= -\mathcal{I}(p(t,x)).
\end{aligned}
\end{equation}
(ii) In \cite[Proposition 4.5]{KST}, Karatzas et al. provided an elegant expression for the rate $D$, i.e. $D(t,x)=-\Big| \frac{\nabla p(t,x)}{p(t,x)}+x \Big|^2$,
for classical SDEs with additive noise, by reversing the time of solution processes. \\
(iii) If the McKean-Vlasov SDE reduces to a classical SDE, our results lead to the expression $D(t,x)=\frac{\Delta p(t,x)}{p(t,x)}+2-\frac{|\nabla p(t,x)|^2}{p(t,x)}-|x|^2$. But in fact, the energy dissipation identities for both cases coincide, with $\mathcal{I}(p)=\int_{\R^d} \big| \frac{\nabla p}{p}+x \big|^2 p \rmd x$. Moreover, our results apply to McKean-Vlasov SDEs with multiplicative noise as well.
\end{rem}

Finally, we compute the energy dissipation identity for the perturbed curve.
To perturb the potential field $\Phi$, we introduce a vector field $\beta$ that satisfies the following assumption.

\begin{ass}\label{perturbed}
The potential field $\beta \in C_c^{\infty}(\R^d, \R^d)$ and $\Phi^{\beta}:=\Phi+\beta$ satisfies Assumption \ref{con}.
\end{ass}

Let us fix $t_0\in [0, T)$.  By replacing $\Phi$ with $\Phi^{\beta}$, we obtain the McKean-Vlasov SDE
\begin{equation*}\label{MVSDE2}
\begin{aligned}
\rmd X^{\beta}(t)&=-\nabla \Phi^{\beta}(X^{\beta}(t)) b\big(\rho(X^{\beta}(t)) \big) \rmd t + \sqrt{\frac{2f\big(\rho(X^{\beta}(t)) \big)}{\rho(X^{\beta}(t)) }} \rmd W^{\beta}(t), t_0\leq t\leq T,\\
X^{\beta}(t_0)&=X(t_0).
\end{aligned}
\end{equation*}
Similarly, the perturbed version of \eqref{NFPE} is given by
\begin{equation}\label{NFPE2}
\begin{aligned}
\partial_t p^{\beta}(t,x) &= {\rm div} \big(\nabla \Phi^{\beta}(x) b\big(p^{\beta}(t,x)\big) p^{\beta}(t,x)\big) + \Delta f\big(p^{\beta}(t,x)\big), (t,x)\in [t_0, T]\times \R^d, \\
p^{\beta}(t_0, x)&=p(t_0,x), x \in \R^d.
\end{aligned}
\end{equation}

Analogously to \eqref{MVSDE} and \eqref{NFPE}, we can provide the relevant definitions for the perturbed versions such as $P^{\beta}(t)$, $W^{\beta}(t)$ and $\theta^{\beta}$.
The semimartingale decomposition of $\theta^{\beta}(t,x)$ is given below.

\begin{thm}\label{solution}
Suppose Assumptions \ref{con}, \ref{perturbed} hold.
Then the energy process $\theta^{\beta}(t,X^{\beta}(t))$ admits the semimartingale decomposition
\begin{equation}\label{decomposition}
\theta^{\beta}(t, X^{\beta}(t)) -\theta^{\beta}(t_0, X^{\beta}(t_0))= M^{\beta}(t) +F^{\beta}(t), t_0\leq t\leq T.
\end{equation}
Here $M^{\beta}(t)$ is the $L^2({\mathbb{P^{\beta}}})$-bounded martingale
\begin{equation*}\label{Mt}
M^{\beta}(t)=\int_{t_0}^t \sqrt{\frac{2f\big(\rho(X^{\beta}(s)) \big)}{\rho(X^{\beta}(s)) }} \nabla \theta^{\beta}(s, X^{\beta}(s)) \rmd W^{\beta}(s)
\end{equation*}
and
\begin{equation}\label{Ft}
\begin{aligned}
F^{\beta}(t)=&\int_{t_0}^t D^{\beta}(s, X^{\beta}(s)) \rmd s \\
:=&\int_{t_0}^t \Big[\phi'\big(\rho(X^{\beta}(s)) \big)\cdot {\rm div} \Big(\nabla\Phi^{\beta}(X^{\beta}(s)) b\big(\rho(X^{\beta}(s)) \big)\rho(X^{\beta}(s)) +\nabla f\big(\rho(X^{\beta}(s)) \big)\Big) \\
 &+\Delta \theta^{\beta} (s, X^{\beta}(s)) \dfrac{f\big(\rho(X^{\beta}(s)) \big)}{\rho(X^{\beta}(s)) }- \Big\langle \nabla \theta^{\beta}(s, X^{\beta}(s)), \nabla\Phi^{\beta}(X^{\beta}(s)) b\big(\rho(X^{\beta}(s)) \big)\Big\rangle \Big] \rmd s
\end{aligned}
\end{equation}
satisfies
\begin{equation*}\label{Ft2}
\begin{aligned}
\mathbb{E}_{\mathbb{P^{\beta}}} [F^{\beta}(t)] =&
-\int_{t_0}^t \mathcal{I}(p^{\beta}(s,x))\rmd s
 -\int_{t_0}^t \mathbb{E}_{\mathbb{P^{\beta}}}\Big[ \big|\nabla \beta(X^{\beta}(s))\big|^2b\big(\rho(X^{\beta}(s)) \big) \\ &
 +2\Big\langle\nabla\big(g\big(\rho(X^{\beta}(s)) \big)+\Phi(X^{\beta}(s))\big), \nabla \beta(X^{\beta}(s)) b\big(\rho(X^{\beta}(s)) \big) \Big \rangle
 \Big] \rmd s
>-\infty.
\end{aligned}
\end{equation*}
\end{thm}

\begin{proof}
According to $\theta^{\beta}(t,x)=\phi(p^{\beta}(t,x))+\Phi(x)$, we obtain
\begin{equation*}\label{50}
\begin{aligned}
\partial_t \theta^{\beta}(t,x) =\phi'(p^{\beta}(t,x)) \cdot {\rm div} \big(\nabla f(p^{\beta}(t,x))+ \nabla \Phi^{\beta}(x) b\big(p^{\beta}(t,x)\big) p^{\beta}(t,x)\big)
\end{aligned}
\end{equation*}
and
\begin{equation*}\label{51}
\begin{aligned}
\rmd  \theta^{\beta}(t, X^{\beta}(t))
=&\partial_t \theta^{\beta} \rmd t +\partial_x \theta^{\beta} \rmd X^{\beta}(t) +\Delta \theta^{\beta}\dfrac{f\big(\rho(X^{\beta}(t))\big)}{\rho(X^{\beta}(t))} \rmd t \\
= &\Big[\phi'\big(\rho(X^{\beta}(t))\big)\cdot {\rm div} \Big(\nabla\Phi^{\beta}(X^{\beta}(t)) b\big(\rho(X^{\beta}(t))\big)\rho(X^{\beta}(t))+\nabla f\big(\rho(X^{\beta}(t))\big)\Big)\\
& +\Delta \theta^{\beta} \frac{f\big(\rho(X^{\beta}(t))\big)}{\rho(X^{\beta}(t))} - \Big \langle \nabla \theta^{\beta}, \nabla\Phi^{\beta}(X^{\beta}(t)) b\big(\rho(X^{\beta}(t))\big )\Big\rangle\Big]\rmd t \\
& + \sqrt{\frac{2f\big(\rho(X^{\beta}(t))\big)}{\rho(X^{\beta}(t))}} \nabla \theta^{\beta}(t, X^{\beta}(t)) \rmd W^{\beta}(t),
\end{aligned}
\end{equation*}
which verifies \eqref{decomposition}.

Now we prove that  $M^{\beta}(t)$ is an $L^2({\mathbb{P^{\beta}}})$-bounded martingale. As
\begin{equation*}\label{52}
\begin{aligned}
 \nabla \theta^{\beta}(t,x)=\phi'(p^{\beta}) \nabla p^{\beta} +\nabla\Phi
 =\dfrac{\int_0^{p^{\beta}}\frac{f'(s)}{b(s)}ds}{(p^{\beta})^2} \nabla p^{\beta} +\nabla\Phi,
\end{aligned}
\end{equation*}
it follows that
\begin{equation}\label{53}
\begin{aligned}
|\nabla \theta^{\beta}|^2 \leq 2\dfrac{\gamma_2^2 |\nabla p^{\beta}|^2}{b_0^2 (p^{\beta})^2} +2|\nabla \Phi|^2.
\end{aligned}
\end{equation}
Substituting \eqref{53} into $\mathbb{E}_{\mathbb{P^{\beta}}}[\langle M^{\beta}(t), M^{\beta}(t)\rangle_T]$ yields
\begin{equation*}\label{54}
\begin{aligned}
 \mathbb{E}_{\mathbb{P^{\beta}}}[\langle M^{\beta}(t), M^{\beta}(t)\rangle_T]
= &\mathbb{E}_{\mathbb{P^{\beta}}} \Big[\int_0^T \dfrac{2f\big(\rho(X^{\beta}(t))\big)}{\rho(X^{\beta}(t))}|\nabla \theta^{\beta}|^2(t, X^{\beta}(t)) \rmd t\Big] \\
\leq & 4\mathbb{E}_{\mathbb{P^{\beta}}} \Big[\int_0^T  \Big(\dfrac{\gamma_2^3 |\nabla \rho(X^{\beta}(t))|^2}{b_0^2 \big(\rho(X^{\beta}(t))\big)^2} +\gamma_2 |\nabla \Phi(X^{\beta}(t))|^2\Big)\rmd t \Big] \\
= &\int_0^T \!\! \int_{\R^d} \Big(\dfrac{\gamma_2^3 |\nabla p^{\beta}|^2}{b_0^2 p^{\beta}} +\gamma_2 |\nabla \Phi|^2p^{\beta}\Big) \rmd x \rmd t.
\end{aligned}
\end{equation*}
Owing to Assumption \ref{con} \ref{phi} and Proposition \ref{P2},
\begin{equation*}\label{int1}
\int_0^T \!\! \int_{\R^d}|\nabla \Phi|^2 p^{\beta}\rmd x \rmd t<+\infty.
\end{equation*}
This calculation and Lemma \ref{integrability} imply
\begin{equation*}\label{55}
\begin{aligned}
\mathbb{E}_{\mathbb{P^{\beta}}}[\langle M^{\beta}(t), M^{\beta}(t)\rangle_T] <+\infty.
\end{aligned}
\end{equation*}

It remains to show that $\mathbb{E}_{\mathbb{P^{\beta}}} [F^{\beta}(t)]>-\infty$.
Let us define
\begin{equation*}\label{D123}
\begin{aligned}
\mathbb{E}_{\mathbb{P^{\beta}}} [F^{\beta}(t)]=\int_{t_0}^t \mathbb{E}_{\mathbb{P^{\beta}}}[D^{\beta}] \rmd t
:=\int_{t_0}^t \mathbb{E}_{\mathbb{P^{\beta}}}[D_1^{\beta}+ D_2^{\beta} +D_3^{\beta}] \rmd t,
\end{aligned}
\end{equation*}
where
\begin{equation*}\label{D1D2D3}
\begin{aligned}
& D_1^{\beta}(t, X^{\beta}(t))=\phi'\big(\rho(X^{\beta}(t))\big)\cdot {\rm div} \Big(\nabla\Phi^{\beta}(X^{\beta}(t)) b\big(\rho(X^{\beta}(t))\big)\rho(X^{\beta}(t))+\nabla f\big(\rho(X^{\beta}(t))\big)\Big), \\
& D_2^{\beta}(t, X^{\beta}(t))=\Delta \theta^{\beta} \dfrac{f\big(\rho(X^{\beta}(t))\big)}{\rho(X^{\beta}(t))}, \\
& D_3^{\beta}(t, X^{\beta}(t))=- \Big\langle \nabla \theta^{\beta}, \nabla\Phi^{\beta}(X^{\beta}(t)) b\big(\rho(X^{\beta}(t))\big)\Big\rangle.
\end{aligned}
\end{equation*}
By applying integration by parts, we obtain
\begin{equation}\label{D123123}
\begin{aligned}
\mathbb{E}_{\mathbb{P^{\beta}}}[D^{\beta}]
=& \mathbb{E}_{\mathbb{P^{\beta}}}[D_1^{\beta}+ D_2^{\beta} +D_3^{\beta}] \\
=&-\int_{\R^d} \nabla\big(\phi'(p^{\beta})p^{\beta}\big)\cdot [\nabla f(p^{\beta}) +\nabla\Phi^{\beta} b(p^{\beta})p^{\beta}] \rmd x \\
& -\int_{\R^d} \langle \nabla f(p^{\beta}), \nabla \theta^{\beta} \rangle \rmd x -\int_{\R^d} \langle \nabla \theta^{\beta}, \nabla \Phi^{\beta} b(p^{\beta}) p^{\beta} \rangle \rmd x \\
=&-\int_{\R^d} \Big\langle \dfrac{\nabla f(p^{\beta})}{b(p^{\beta}) p^{\beta}} +\nabla \Phi, \nabla f(p^{\beta}) +\nabla\Phi^{\beta} b(p^{\beta}) p^{\beta} \Big\rangle \rmd x \\
=&-\mathcal{I}(p^{\beta}) -\int_{\R^d} \Big\langle\nabla\big(g(p^{\beta})+\Phi\big), \nabla \beta b(p^{\beta})p^{\beta} \Big\rangle \rmd x.
\end{aligned}
\end{equation}
Using Proposition \ref{P2} and Lemma \ref{integrability}, we compute
\begin{equation}\label{int2}
\begin{aligned}
\mathcal{I}(p^{\beta}) &=\int_{\R^d} |\nabla g(p^{\beta})+\nabla \Phi|^2 b(p^{\beta}) p^{\beta} \rmd x \\
&\leq \dfrac{2 \gamma_2^2 b_1}{b_0^2}\int_{\R^d} \dfrac{|\nabla p^{\beta}|^2}{p^{\beta}} \rmd x +2b_1\int_{\R^d} |x|^2 p^{\beta} \rmd x<+ \infty.
\end{aligned}
\end{equation}
This calculation and Assumption \ref{perturbed} yield
\begin{equation*}
\begin{aligned}
\mathbb{E}_{\mathbb{P^{\beta}}} [F^{\beta}(t)]=& \int_{t_0}^t \mathbb{E}_{\mathbb{P^{\beta}}}[D^{\beta}] \rmd t \\
=&-\int_{t_0}^t \mathcal{I}(p^{\beta}(s,x))\rmd s -\int_{t_0}^t \mathbb{E}_{\mathbb{P^{\beta}}}\Big[ \big|\nabla \beta(X^{\beta}(s))\big|^2b\big(\rho(X^{\beta}(s)) \big) \\
& +2\Big\langle\nabla\big(g\big(\rho(X^{\beta}(s)) \big)+\Phi(X^{\beta}(s))\big), \nabla \beta(X^{\beta}(s)) b\big(\rho(X^{\beta}(s)) \big) \Big \rangle
 \Big] \rmd s
>-\infty.
\end{aligned}
\end{equation*}
\end{proof}

Similar to the unperturbed curve $(P(t))_{0\leq t\leq T}$ of \eqref{MVSDE}, we can establish the trajectorial rate of relative entropy dissipation and the energy dissipation identity as follows.

\begin{cor}
Suppose Assumptions \ref{con}, \ref{perturbed} hold. For all $t_0 \in [0,T)$, the trajectorial rate of relative entropy dissipation identity is given by
\begin{equation}\label{56}
\begin{aligned}
\lim_{t \downarrow t_0} \dfrac{\mathbb{E}_{\mathbb{P^{\beta}}}[\theta^{\beta}(t, X^{\beta}(t))| \mathcal{G}(t_0)] -\theta^{\beta}(t_0, X^{\beta}(t_0)) }{t-t_0}
= D^{\beta}(t_0, X^{\beta}(t_0)),
\end{aligned}
\end{equation}
where the limit exists in $L^1(\mathbb{P^{\beta}})$.
\end{cor}

\begin{proof}
According to \eqref{decomposition} and \eqref{Ft}, we obtain
\begin{equation}\label{57}
\begin{aligned}
\mathbb{E}_{\mathbb{P^{\beta}}}[\theta^{\beta}(t, X^{\beta}(t))| \mathcal{G}(t_0)] -\theta^{\beta}(t_0, X^{\beta}(t_0))
=\mathbb{E}_{\mathbb{P^{\beta}}}\Big[\int_{t_0}^t D^{\beta}(s, X^{\beta}(s)) \rmd s \Big| \mathcal{G}(t_0) \Big].
\end{aligned}
\end{equation}
Similar to the proof of \eqref{D123123}, we have
\begin{equation*}\label{58}
\begin{aligned}
\mathbb{E}_{\mathbb{P^{\beta}}}[|D^{\beta}|]< +\infty
\end{aligned}
\end{equation*}
based on Assumptions \ref{con}, \ref{perturbed}, Proposition \ref{P2} and Lemma \ref{integrability}.
Thus, by the Lebesgue differentiation theorem, we obtain
\begin{equation}\label{59}
\begin{aligned}
\lim_{t \downarrow t_0} \mathbb{E}_{\mathbb{P^{\beta}}} \bigg[\dfrac{\int_{t_0}^t D^{\beta}(s, X^{\beta}(s)) \rmd s}{t-t_0} \bigg]
=\mathbb{E}_{\mathbb{P^{\beta}}} [D^{\beta}(t_0, X^{\beta}(t_0))].
\end{aligned}
\end{equation}
Furthermore, Scheff\'{e}'s lemma \cite[Theorem 5.10]{W} states that \eqref{59} is equivalent to
\begin{equation*}\label{60}
\begin{aligned}
\lim_{t \downarrow t_0} \bigg\| \dfrac{\int_{t_0}^t D^{\beta}(s, X^{\beta}(s)) \rmd s}{t-t_0} -D^{\beta}(t_0, X^{\beta}(t_0)) \bigg\|_{L^1(\mathbb{P^{\beta}})} =0.
\end{aligned}
\end{equation*}
The tower property $\mathbb{E}(X)=\mathbb{E}(\mathbb{E}(X|\mathcal{G}(t_0)))$ implies that
\begin{equation}\label{61}
\begin{aligned}
& \lim_{t \downarrow t_0} \bigg\| \mathbb{E}_{\mathbb{P^{\beta}}} \bigg[\dfrac{\int_{t_0}^t D^{\beta}(s, X^{\beta}(s)) \rmd s}{t-t_0} \bigg|\mathcal{G}(t_0) \bigg] -D^{\beta}(t_0, X^{\beta}(t_0)) \bigg\|_{L^1(\mathbb{P^{\beta}})} \\
= & \lim_{t \downarrow t_0} \bigg\| \dfrac{\int_{t_0}^t D^{\beta}(s, X^{\beta}(s)) \rmd s}{t-t_0} -D^{\beta}(t_0, X^{\beta}(t_0)) \bigg\|_{L^1(\mathbb{P^{\beta}})}=0.
\end{aligned}
\end{equation}
Employing \eqref{57} in \eqref{61} finishes the $L^1(\mathbb{P^{\beta}})$-convergence of \eqref{56}.
\end{proof}

\begin{cor}
Suppose Assumptions \ref{con}, \ref{perturbed} hold. Then the relative entropy dissipation is given by
\begin{equation}\label{id2}
\begin{aligned}
& \dfrac{\rmd }{\rmd t}H_g(P^{\beta}(t)|P_{\infty}) \\
=& -I_g(P^{\beta}(t)|P_{\infty}) -\mathbb{E}_{\mathbb{P^{\beta}}}\Big[ \Big\langle\nabla\Big(g\big(\rho(X^{\beta}(t))\big)+\Phi(X^{\beta}(t))\Big), \nabla \beta(X^{\beta}(t)) b\big(\rho(X^{\beta}(t))\big)  \Big\rangle \Big].
\end{aligned}
\end{equation}
\end{cor}

\begin{proof}
Using the property of the expectation of the $L^2({\mathbb{P^{\beta}}})$-bounded martingale $M^{\beta}(t)$, we have
\begin{equation*}\label{id1}
\begin{aligned}
& H_g(P^{\beta}(t)|P_{\infty})-H_g(P^{\beta}(t_0)|P_{\infty}) \\
=&-\int_{t_0}^t I_g(P^{\beta}(s)|P_{\infty}) \rmd s -\int_{t_0}^t \mathbb{E}_{\mathbb{P^{\beta}}}\Big[ \Big\langle\nabla\Big(g\big(\rho(X^{\beta}(s))\big)+\Phi(X^{\beta}(s))\Big), \nabla \beta(X^{\beta}(s)) b\big(\rho(X^{\beta}(s))\big)  \Big \rangle \Big] \rmd s
\end{aligned}
\end{equation*}
by taking the expectation of \eqref{decomposition}.
Applying the Lebesgue differentiation theorem to $H_g(P^{\beta}(t)|P_{\infty})$, we conclude \eqref{id2}.
\end{proof}

\begin{rem}
Correspondingly, we have the energy dissipation
\begin{equation}\label{edi2}
\begin{aligned}
& \dfrac{\rmd }{\rmd t}\mathcal{F}(p^{\beta}(t,x)) \\
=& -\mathcal{I}(p^{\beta}(t,x))
 -\int_{\R^d} \Big\langle\nabla\big(g\big(p^{\beta}(s,x)\big)+\Phi(x)\big), \nabla \beta(x) b\big(p^{\beta}(s,x)\big)p^{\beta}(s,x) \Big \rangle \rmd x.
\end{aligned}
\end{equation}
\end{rem}

\subsection{Wasserstein metric derivative}\label{wh}

In this subsection, we first introduce the modified Wasserstein metric and its associated properties \cite{CLSS,DNS,LMS}, followed by the modified Wasserstein metric derivative.
Denote the vector field
\begin{equation}\label{v}
v^{\beta}(t,x):=-\nabla\big(g(p^{\beta}(t,x))+\Phi^{\beta}(x)\big),
\end{equation}
then the perturbed version of the nonlinear Fokker-Planck equation \eqref{NFPE2} can be written as
\begin{equation*}\label{ce}
\begin{aligned}
\partial_t p^{\beta}(t,x)&+{\rm div} \big(v^{\beta}(t,x)  h(p^{\beta}(t,x)) \big) =0, (t,x) \in [t_0, T] \times \R^d, \\
p^{\beta}(t_0, x)&=p(t_0, x), x\in \R^d.
\end{aligned}
\end{equation*}

Benamou-Brenier formula illustrates that the quadratic Wasserstein metric $W_2$ can be given via the continuity equation, that is
\begin{equation*}\label{quadratic}
\begin{aligned}
W_2^2(\bar{\mu}_0, \bar{\mu}_1)&:= \inf \Big\{ \int_0^1 \!\! \int_{\R^d} |v_t|^2 u_t \rmd x \rmd t \Big|  \partial_t \mu_t+ {\rm div}\big(v_t \mu_t \big)=0 \\ & \text{ holds in the distributional sence, }
\rmd \mu_t=u_t\rmd x, \mu_i=\bar{\mu}_i, i=0,1 \Big\}
\end{aligned}
\end{equation*}
for $\bar{\mu}_0, \bar{\mu}_1 \in \mathcal{P}_2(\R^d)$.

We now consider the impact of mobility and provide the definitions of the modified Wasserstein metric and the modified Wasserstein scalar product.

\begin{de}[Modified Wasserstein metric]\label{modified}
For an increasing concave function $h:[0, \infty) \rightarrow [0, \infty)$, given $\bar{\mu}_i \in \mathcal{P}_2(\R^d)$, $i=0,1$, we define {\em the modified Wasserstein metric} by
\begin{equation*}\label{metric}
W_h^2(\bar{\mu}_0, \bar{\mu}_1):=\inf \Big\{ \int_0^1 \!\! \int_{\R^d} |v_t|^2 h(u_t) \rmd x \rmd t \Big| (\mu, v) \in \mathcal{CE}(0,1), \rmd \mu_t=u_t\rmd x, \mu_i=\bar{\mu}_i, i=0,1 \Big\},
\end{equation*}
where $\mathcal{CE}(0,1)=\Big\{(\mu, v) \Big| \partial_t \mu_t+{\rm div} \big(v_t \hat{h}(\mu_t) \big) =0  \text{ holds in the distributional sence}\Big \}$.
\end{de}

\begin{de}[Modified Wasserstein scalar product]\label{scalar}
For two tangent vectors $s_1$, $s_2 : \R^d \rightarrow \R$, we define {\em the modified Wasserstein scalar product} at $\mu$ by
\begin{equation*}\label{product}
\langle s_1, s_2 \rangle_{\hat{h}(\mu)} :=\int_{\R} \hat{h}(\mu) \nabla \psi_1 \cdot \nabla \psi_2 \rmd x,
\end{equation*}
where $-{\rm div} \big(\hat{h}(\mu)\nabla \psi_i\big)=s_i$, $i=1,2$.
\end{de}

\begin{rem}
(i) If $\mu=u \rmd x$ and $u$ is uniformly bounded, then the condition of Definition \ref{modified} can be relaxed to require that $h$ is concave \cite{CLSS, LM}. \\
(ii) Since $W_h$ can take the value $+\infty$, the modified Wasserstein metric is a pseudo-metric. However, by imposing the conditions $\bar{\mu}_i \in \mathcal{P}_{ac, 2}(\R^d)$, $\bar{\mu}_i=u_i \rmd x$ and $0 \leq u_i \leq M$, $i=0,1$, the finiteness of the metric can be ensured \cite[Theorem 3]{LM}. \\
(iii) The norm of the tangent vector $s$ can be expressed as the tensor product
\begin{equation*}
\| s \|_{\hat{h}(\mu)}^2 =\int_{\R} \hat{h}(\mu) | \nabla \psi |^2 \rmd x,
\end{equation*}
where $-{\rm div} (\hat{h}(\mu)\nabla \psi)=s$.  Furthermore, the modified Wasserstein metric can be defined in terms of the norm of the tangent vector as follows:
$$W_h^2(\bar{\mu}_0, \bar{\mu}_1):=\inf \Big\{ \int_0^1 \bigg \| \dfrac{\partial \mu_t }{\partial t} \bigg \|^2_{\hat{h}(\mu_t)}  \rmd t \Big| (\mu, v) \in \mathcal{CE}(0,1) ,  \mu_i=\bar{\mu}_i, i=0,1 \Big\}.$$
\end{rem}

Similar to $W_2$, the Wasserstein metric derivative of an absolutely continuous curve with respect to $W_h$ can also be calculated \cite[Corollary 5.20, Theorem 5.21]{DNS}.

\begin{lem}\label{nabla}
Let $\mu_t$ is an absolutely continuous curve with respect to $W_h$, $\mu \ll \rmd x$ and $(\mu, v) \in \mathcal{CE}(0,T)$.
Suppose $h$ is increasing and concave, and that $\lim_{r \downarrow 0} h(r)=0$, $\lim_{r\uparrow +\infty}r^{-1}h(r)<\widetilde{M}$ for some constant $\widetilde{M}$.
If $\mu \in \mathcal{P}_2(\R^d)$, then
$$ \big| \mu'_{t_0} \big| := \lim_{t \rightarrow t_0} \dfrac{W_h\big(\mu_t, \mu_{t_0}\big)}{|t-t_0|}
=\Big( \int_{\R^d} |v_{t_0}|^2 \hat{h}(\mu_{t_0}) \rmd x \Big)^\frac{1}{2}, \quad  \mathcal{L}^1 \text{-a.e. } t_0\in (0,T)$$
is equivalent to
$v_{t_0}\in \overline{\{\nabla \zeta : \zeta \in C_c^{\infty}(\R^d) \}}^{L^2_{\hat{h}(\mu_{t_0})}(\R^d, \R^d)}=:{\rm Tan}_{\hat{h}(\mu_{t_0})}\mathcal{P}_{2}(\R^d)$.
\end{lem}

To satisfy the conditions of Lemma \ref{nabla}, we must impose the following assumption.

\begin{ass}\label{tangent}
We assume that
$$v(t, X(t))\in {\rm Tan}_{\hat{h}(P(t))}\mathcal{P}_{2}(\R^d),$$
where $v(t,x):=-\nabla \big ( g(p(t,x)) +\Phi(x) \big)$.
\end{ass}

\begin{thm}\label{derivative}
Suppose $h$ is increasing and concave, and Assumptions \ref{con}, \ref{perturbed}, \ref{tangent} hold. Then the Wasserstein metric derivative of the perturbed curve $t \mapsto P^{\beta}(t)$ is given by
\begin{equation}\label{wh2}
\begin{aligned}
\lim_{t \downarrow t_0} \dfrac{W_h\big(P^{\beta}(t), P^{\beta}(t_0)\big)}{t-t_0}
&= \big\| \nabla(g\big(\rho(X^{\beta}(t_0))\big)+\Phi^{\beta}(X(t_0))) \big\|_{L^2(\hat{h}(P^{\beta}(t_0)))} \\
&=\Big( \int_{\R^d} \big|\nabla\big(g(p^{\beta}(t_0,x))+\Phi^{\beta}(x)\big)\big|^2 h\big(p^{\beta}(t_0,x)\big) \rmd x \Big)^\frac{1}{2} \\
&= \big(\mathcal{I}(p^{\beta}(t_0))\big)^{\frac{1}{2}}.
\end{aligned}
\end{equation}
\end{thm}

\begin{proof}
In view of \eqref{int2} and Assumption \ref{perturbed}, we conclude
$$\int_0^T \big\| \nabla(g\big(\rho(X^{\beta}(t))\big)+\Phi^{\beta}(X(t))) \big\|_{L^2(\hat{h}(P^{\beta}(t)))} \rmd t <+\infty.$$
Therefore, the curve $t \mapsto P^{\beta}(t)$ is absolutely continuous with respect to $W_h$ \cite[Def 1.1.1]{AGS}.
It is obvious that $\lim_{r \downarrow 0} h(r)= \lim_{r \downarrow 0} rb(r)=0$.
As $b(s)\leq b_1$, it follows that  $\lim_{r\uparrow +\infty}r^{-1}h(r)<\widetilde{M}$ for some constant $\widetilde{M}$. Using \eqref{v} and Lemma \ref{nabla}, we can derive the Wasserstein metric derivative \eqref{wh2}.
\end{proof}

Let the potential field $\beta =0$, we can derive the unperturbed version of the Wasserstein metric derivative.

\begin{thm}\label{derivative0}
Suppose $h$ is increasing and concave, and Assumptions \ref{con}, \ref{tangent} hold. Then the Wasserstein metric derivative of the curve $t \mapsto P(t)$ is given by
\begin{equation}\label{wh20}
\begin{aligned}
\lim_{t \downarrow t_0} \dfrac{W_h\big(P(t), P(t_0)\big)}{t-t_0}
&= \big\| \nabla\big(g(\rho(X(t_0)))+\Phi(X(t_0))\big) \big\|_{L^2(\hat{h}(P(t_0)))} \\
&=\Big( \int_{\R^d} \big|\nabla\big(g(p(t_0,x))+\Phi(x)\big)\big|^2 h\big(p(t_0,x)\big) \rmd x \Big)^\frac{1}{2} \\
&= \big(\mathcal{I}(p(t_0))\big)^{\frac{1}{2}}.
\end{aligned}
\end{equation}
\end{thm}

\subsection{Gradient Flow}\label{gf}

Similar to the classical $W_2$-gradient flow result \cite{AGS}, we now provide the definition of the $W_h$-gradient flow. Based on this definition, we present the gradient flow result for the free energy functional $\mathcal{F}$. Here, we adopt the notation $\mathcal{F}(P(t)):=\mathcal{F}(p(t,x))$.

\begin{de}[Gradient flow]\label{defgf}
Let $\mathcal{F}:\mathcal{P}_2(\R^d) \rightarrow \R $ is a functional. Suppose that there exists a unique function $\mathbb{G}_{\bar{\mu}} :\R^d \rightarrow \R $ such that
\begin{equation}\label{gradient}
\Big \langle \mathbb{G}_{\bar{\mu}}, \dfrac{\partial \mu_t}{\partial t}\Big|_{t=0} \Big \rangle_{\hat{h}(\bar{\mu})} =\dfrac{\rmd }{\rmd t}\bigg|_{t=0} \mathcal{F}(\mu_t)
\end{equation}
for any smooth curve $\mu_t: (-t_0, t_0) \rightarrow \mathcal{P}_2(\R^d)$ with $\mu_0=\bar{\mu}$.
We then call
$${\rm grad}_{W_h} \mathcal{F}(\bar{\mu}):=\mathbb{G}_{\bar{\mu}}$$
the {\em gradient flow} of $\mathcal{F}$ with respect to $W_h$ at $\bar{\mu}$.
\end{de}

According to Definitions \ref{scalar} and \ref{defgf}, we can express the following form:
\begin{equation}\label{grad}
{\rm grad}_{W_h} \mathcal{F}(\mu_0)=-{\rm div}\Big(\hat{h}(\mu_0) \nabla\big(\dfrac{\delta \mathcal{F}(\mu_0) }{\delta \mu}\big)\Big),
\end{equation}
where $\dfrac{\delta \mathcal{F}(\mu_0) }{\delta \mu}$ is the first $L^2$-variation.
Indeed, if $\partial_t \mu_t={\rm div}\big(v_t \hat{h}(\mu_t)\big)$, then
\begin{equation*}
\begin{aligned}
\dfrac{\rmd}{\rmd t}\bigg|_{t=0}\mathcal{F}(\mu_t)
=&\int_{\R^d}\dfrac{\delta \mathcal{F}}{\delta \mu}(x) \dfrac{\partial \mu_t}{\partial t}\bigg|_{t=0} \rmd x \\
=&-\int_{\R^d}\dfrac{\delta \mathcal{F}}{\delta \mu}(x) {\rm div}(\hat{h}(\mu_0)v_t)\rmd x
=\int_{\R^d}\nabla\Big(\dfrac{\delta \mathcal{F}}{\delta \mu}(x)\Big) \cdot v_t \hat{h}(\mu_0) \rmd x.
\end{aligned}
\end{equation*}
In view of \eqref{product} and \eqref{gradient}, we have \eqref{grad}.

Now, we employ
\begin{equation*}\label{slope1}
\begin{aligned}
\big|\nabla_{W_h} \mathcal{F} \big|(P(t_0)):=\lim_{t \downarrow t_0} \dfrac{H_g(P(t)|P_{\infty}) -H_g(P(t_0)|P_{\infty})}{W_h\big(P(t), P(t_0)\big)}
\end{aligned}
\end{equation*}
to compute the Wasserstein metric slope of the free energy functional $\mathcal{F}$ along the curve $t \mapsto P(t)$.
Denote the $\R^d$-valued random vectors by
\begin{equation*}\label{a}
\mathbf{a}:=\nabla\big(g(\rho(X(t_0)))+\Phi(X(t_0))\big), \mathbf{b}:=\nabla \beta(X(t_0)).
\end{equation*}

\begin{thm}\label{main}
Suppose $h$ is increasing and concave, and Assumptions \ref{con}, \ref{perturbed}, \ref{tangent} hold.
Then the gradient flow of $\mathcal{F}$ with respect to $W_h$ is given by
\begin{equation}\label{form2}
{\rm grad}_{W_h} \mathcal{F}(P(t))=-{\rm div}\Big( h\big(\rho(X(t))\big) \big(\nabla g ( \rho(X(t)) ) +\nabla \Phi(X(t)) \big) \Big).
\end{equation}
\end{thm}

\begin{proof}
Invoking \eqref{edi2} and \eqref{wh2}, we compute
the Wasserstein metric slope of the free energy functional $\mathcal{F}$ along the curve $t \mapsto P^{\beta}(t)$, which is equal to
\begin{equation}\label{main2}
\big|\nabla_{W_h} \mathcal{F} \big|(P^{\beta}(t_0)) = -\bigg \langle \mathbf{a},  \dfrac{\mathbf{a} +\mathbf{b} }{\big\| \mathbf{a}+\mathbf{b} \big\|_{L^2(\hat{h}(P(t_0)))}}
\bigg \rangle_{L^2(\hat{h}(P(t_0)))}.
\end{equation}
Let $\beta=0$ in \eqref{main2}, we obtain
the Wasserstein metric slope of the free energy functional $\mathcal{F}$ along the curve $t \mapsto P(t)$, which is given by
\begin{equation*}\label{main1}
\big|\nabla_{W_h} \mathcal{F} \big|(P(t_0)) = -\big\| \mathbf{a} \big\|_{L^2(\hat{h}(P(t_0)))}.
\end{equation*}
It follows then from the Cauchy-Schwarz inequality that
\begin{equation*}\label{main3}
\big|\nabla_{W_h} \mathcal{F} \big|(P^{\beta}(t_0)) -\big|\nabla_{W_h} \mathcal{F} \big|(P(t_0)) \geq 0
\end{equation*}
holds, and it equals $0$ if and only if $\mathbf{a}+\mathbf{b}$ is a positive multiple of $\mathbf{a}$.
The arbitrariness of $\beta$ indicates \eqref{form2} holds
and
\begin{equation}\label{form1}
\begin{aligned}
\big|{\rm grad}_{W_h} \mathcal{F}(P(t))\big|^2& = \big\langle {\rm grad}_{W_h} \mathcal{F}(P(t)), {\rm grad}_{W_h} \mathcal{F}(P(t)) \big\rangle_{\hat{h}(P(t))} \\
& =\int_{\R^d} | \nabla g(p(t,x))+\nabla \Phi(x) |^2 h(p(t,x)) \rmd x.
\end{aligned}
\end{equation}
\end{proof}

\section{An example}\label{example}

In this section, we apply our results to the Fermi-Dirac-Fokker-Planck equation.
Additionally, we perform numerical simulations to observe the energy dissipation rate.

If $b(p)=1-p$, $f(p)=p$ and $\Phi(x)=\frac{|x|^2}{2}$, the McKean-Vlasov SDE is given by
\begin{equation}\label{sde}
\begin{aligned}
\rmd X(t)&= -X(t) \big(1-\rho(X(t))\big) \rmd t + \sqrt{2} \rmd W(t),  t \in [0, T], \\
X(0)&=\xi_0
\end{aligned}
\end{equation}
and the corresponding nonlinear Fokker-Planck equation is given by
\begin{equation}\label{FD}
\begin{aligned}
 \partial_t p(t,x) &= {\rm div} (x (1-p(t,x)) p(t,x)) + \Delta p(t,x), (t,x)\in[0,T]\times \R^d, \\
 p(0, x)&=p_0(x), x\in \R^d.
\end{aligned}
\end{equation}

The existence and uniqueness of the strong solution to \eqref{sde} are guaranteed by Theorem 2.1 in \cite{W2}.
The existence and uniqueness of the smooth solution to \eqref{FD}, along with the properties outlined in Assumption \ref{con} \ref{smooth} can be found in \cite{CLR}.
The stationary solution of \eqref{FD} is given by the form
\begin{equation}\label{infty}
p_{c, \infty}=\dfrac{1}{1+c\rme^{\frac{|x|^2}{2}}}, c\geq 0.
\end{equation}
Since the mass is preserved, we can find a unique constant $c_0$ such that $\int_{\R^d} p_{c_0, \infty} \rmd x=\int_{\R^d} p_0 \rmd x =1$.
We choose an initial density function $p_0$ such that both Assumption \ref{con} \ref{initial} and $p_{c_2, \infty}\leq p_0\leq p_{c_1, \infty}$ hold, where $0<c_1 <c_2$.
Since the solution of \eqref{FD} satisfies the comparison principle, it follows that $p_{c_2, \infty} \leq p(t, \cdot)\leq p_{c_1, \infty}$ for all $t\in [0,T]$, and furthermore, $0 \leq p\leq \frac{1}{1+c_1}$.
Consequently Assumption \ref{con} \ref{b} is satisfied for $b_0:=\frac{c_1}{1+c_1}$ and $b_1:=1$.

Considering that $0$ and $1$ are singularity point of $\frac{1}{p(1-p)}$, we use the form $g(s)=\int_{b_1}^s \frac{1}{t(1-t)} \rmd t$. Thus
$$g(s)=\ln s-\ln (1-s) +\ln\dfrac{1-b_1}{b_1}, \eta(r)=r\ln r +(1-r)\ln (1-r)+\big(\ln \dfrac{1-b_1}{b_1}\big)r$$
and
$$\mathcal{F}(\rho)=\int_{\R^d}\Big(\rho\ln\rho +(1-\rho)\ln(1-\rho) +\big(\ln \dfrac{1-b_1}{b_1}\big)\rho +\dfrac{|x|^2}{2}\rho \Big)\rmd x.$$
Note that $g^{-1}(x)=\frac{1}{1+\frac{1-b_1}{b_1}\rme^{-x}}$, so we have $p_{\infty}:=g^{-1}(-\Phi+c)=\frac{1}{1+\frac{1-b_1}{b_1}\rme^{-c}\rme^{\frac{|x|^2}{2}}}$.
Obviously, the definition of $p_{\infty}$ is consistent with \eqref{infty} by taking $c_0=\frac{1-b_1}{b_1}e^{-c}$.

Using \eqref{560} and \eqref{edi20}, we compute the trajectorial rate of relative entropy dissipation
\begin{equation}\label{1p}
\begin{aligned}
& \lim_{t \downarrow t_0} \dfrac{\mathbb{E}_{\mathbb{P}}[\theta(t, X(t))| \mathcal{G}(t_0)] -\theta(t_0, X(t_0)) }{t-t_0}
= D(t_0, X(t_0)) \\
= &\ln\big(1-\rho(X(t_0))\big)\bigg( 1-\dfrac{1}{\rho(X(t_0))}+\dfrac{\langle X(t_0), \nabla \rho(X(t_0)) \rangle}{\rho(X(t_0))} -\dfrac{2\Delta \rho(X(t_0))}{\big(\rho(X(t_0))\big)^2} \\
&+\dfrac{2|\nabla \rho(X(t_0))|^2}{\big(\rho(X(t_0))\big)^3} \bigg) +\dfrac{|\nabla \rho(X(t_0))|^2}{\big(\rho(X(t_0))\big)^2\big(1-\rho(X(t_0))\big)} -|X(t_0)|^2\big(1-\rho(X(t_0))\big)+1
\end{aligned}
\end{equation}
in $L^1(\mathbb{P})$
and the energy dissipation
\begin{equation*}
\begin{aligned}
 \dfrac{\rmd }{\rmd t}\mathcal{F}(p(t,x))=-\mathcal{I}(p(t,x))
= - \int_{\R^d} \bigg| \dfrac{\nabla p(t,x)}{p(t,x)\big(1-p(t,x)\big)}+x \bigg|^2\big(1-p(t,x)\big)p(t,x) \rmd x,
\end{aligned}
\end{equation*}
where $\theta(t, x)=\ln p(t,x) +\frac{1-p(t,x)}{p(t,x)}\ln \big(1-p(t,x)\big) +\frac{|x|^2}{2}+\ln\frac{1-b_1}{b_1}$.

We now verify Assumption \ref{tangent}. Set $B_n$ to be the open ball centered at the origin with radius $n$ and define
\begin{equation*}\label{zeta}
\zeta_n(t, x):=
\begin{cases}
 g(p(t,x))+\Phi(x), \quad & |x|<n, \\
 \tilde{\zeta}(x), &  n \leq |x| \leq n+1, \\
 0, & |x|>n+1,
\end{cases}
\end{equation*}
where $\tilde{\zeta}(x) \in C^{\infty}(\overline{B_{n+1} \backslash B_n})$ and $| \nabla \tilde{\zeta}(x) |< 2\big|g(p(t,x))+\Phi(x)\big| $. Then
\begin{equation*}\label{L2}
\begin{aligned}
&\int_{\R^d} |\nabla \zeta_n-\nabla \big(g(p)+\Phi \big) |^2 p b(p) \rmd x \\
= & \int_{B_n \cup \overline{(B_{n+1} \backslash B_n)} \cup B_{n+1}^c} |\nabla \zeta_n-\nabla (g(p)+\Phi \big) |^2 p b(p) \rmd x \\
\leq & \int_{B_{n+1}^c}|\nabla \big(g(p)+\Phi \big) |^2 pb(p) \rmd x + \int_{\overline{(B_{n+1} \backslash B_n)}} 2|\nabla \zeta_n|^2 pb(p)\rmd x \\
& +\int_{\overline{(B_{n+1} \backslash B_n)}} 2|\nabla \big(g(p)+\Phi\big)|^2 pb(p)\rmd x \\
\leq & \int_{B_n^c}2|\nabla \big(g(p)+\Phi \big) |^2 pb(p) \rmd x + \int_{B_n^c} 8| g(p)+\Phi|^2 pb(p) \rmd x \\
\leq & \int_{B_n^c}2|\nabla \big(g(p)+\Phi \big) |^2 pb(p) \rmd x + \int_{B_n^c} 16| g(p)|^2 pb(p) \rmd x +\int_{B_n^c} 16| \Phi|^2 pb(p) \rmd x.
\end{aligned}
\end{equation*}
Recalling \eqref{int2}, we discover
\begin{equation}\label{e1}
\int_{\R^d} |\nabla \big(g(p)+\Phi \big) |^2 pb(p) \rmd x < +\infty.
\end{equation}
In view of $\Phi(x)=\frac{|x|^2}{2}$ and $p(t,x)\leq \frac{1}{1+c_1\rme^{\frac{|x|^2}{2}}}$, we obtain
\begin{equation}\label{e2}
\int_{\R^d} 16| \Phi|^2 pb(p) \rmd x \leq \int_{\R^d} |x|^4 \dfrac{ b_1}{1+c_1\rme^{\frac{|x|^2}{2}}} \rmd x <+\infty.
\end{equation}
According to the fact $\frac{1}{1+c_2\rme^{\frac{|x|^2}{2}}} \leq p(t,x)\leq \frac{1}{1+c_1\rme^{\frac{|x|^2}{2}}}$, we conclude that
\begin{equation}\label{e3}
\begin{aligned}
& \int_{\R^d} 16| g(p)|^2 pb(p) \rmd x \leq \int_{\R^d} 16 \dfrac{b_1 \gamma^2}{b_0^2}|\ln p|^2 p \rmd x \\
\leq & \int_{\{ p(t,x) \geq 1\}} 16 \dfrac{b_1 \gamma^2}{b_0^2} \big|\ln (1+c_1\rme^{\frac{|x|^2}{2}})\big|^2 \dfrac{1}{1+c_1\rme^{\frac{|x|^2}{2}}} \rmd x\\
&+ \int_{\{p(t,x)<1\} } 16 \dfrac{b_1 \gamma^2}{b_0^2} \big|\ln (1+c_2\rme^{\frac{|x|^2}{2}})\big|^2 \dfrac{1}{1+c_1\rme^{\frac{|x|^2}{2}}} \rmd x <+\infty.
\end{aligned}
\end{equation}
Employing \eqref{e1}, \eqref{e2} and \eqref{e3}, we deduce
\begin{equation*}
\int_{\R^d} |\nabla \zeta_n-\nabla \big(g(p)+\Phi \big) |^2 p b(p) \rmd x \rightarrow 0, \text{ as } n\rightarrow \infty,
\end{equation*}
which implies that
\begin{equation*}
-\nabla\big(g(\rho(X(t)))+\Phi(X(t))\big) \in {\rm Tan}_{\hat{h}(P(t))}\mathcal{P}_{2}(\R^d).
\end{equation*}

Since $p$ is uniformly bounded, $b$ is defined on the bounded domain $\big[\frac{c_1}{1+c_1}, 1\big]$. Therefore, $h'(r)\geq 0$ is no longer required for Theorem \ref{derivative0}.
Now we are in the position to apply Theorems \ref{derivative0} and \ref{main}.
Using \eqref{wh20}, \eqref{form2} and \eqref{form1}, we compute that
\begin{equation*}
\begin{aligned}
\lim_{t \downarrow t_0} \dfrac{W_h\big(P(t), P(t_0)\big)}{t-t_0}
&= \bigg\| \dfrac{\nabla \rho(X(t_0))}{\rho(X(t_0)) \big(1-\rho(X(t_0)) \big)}+ X(t_0) \bigg\|_{L^2(\hat{h}(P(t_0)))} \\
&= \big(\mathcal{I}(p(t_0))\big)^{\frac{1}{2}},
\end{aligned}
\end{equation*}
\begin{equation*}
{\rm grad}_{W_h} \mathcal{F}(P(t_0))=-{\rm div}\Big( \rho(X(t_0))\big(1-\rho(X(t_0))\big) \Big (\dfrac{\nabla \rho(X(t_0))}{\rho(X(t_0)) \big(1-\rho(X(t_0)) \big)}+ X(t_0) \Big) \Big)
\end{equation*}
and
\begin{equation*}
\begin{aligned}
\big|{\rm grad}_{W_h} \mathcal{F}(P(t_0))\big|
&= \bigg\| \dfrac{\nabla \rho(X(t_0))}{\rho(X(t_0)) \big(1-\rho(X(t_0)) \big)}+ X(t_0) \bigg\|_{L^2(\hat{h}(P(t_0)))} \\
&= -\Big(\int_{\R^d} \bigg| \dfrac{\nabla p(t_0,x)}{p(t_0,x)\big(1-p(t_0,x)\big)}+x \bigg|^2\big(1-p(t_0,x)\big)p(t_0,x) \rmd x \Big)^{\frac{1}{2}},
\end{aligned}
\end{equation*}
where $\rmd \hat{h}(P(t_0))=p(t_0,x)(1-p(t_0,x))\rmd x$.

Therefore, the solution of \eqref{FD} represents the gradient flow of the free energy functional $\mathcal{F}(\rho)$ with respect to the modified Wasserstein metric $W_h$.
Based on \eqref{1p}, we plot the energy dissipation $\mathcal{F}_i(p_t)-\mathcal{F}_i(p_0)$ over time for 500 trajectories, as well as the average energy dissipation $\overline{\mathcal{F}}(p_t)-\overline{\mathcal{F}}(p_0)$.
As shown in Fig.~\ref{figtraj1} and Fig.~\ref{figenergy1}, the energy of the system dissipates at an exponential rate, which coincides with the result in \cite{CLR}.

\begin{figure}[ht]
    \centering
    \begin{minipage}[b]{0.48\textwidth}
        \centering
        \includegraphics[width=\textwidth]{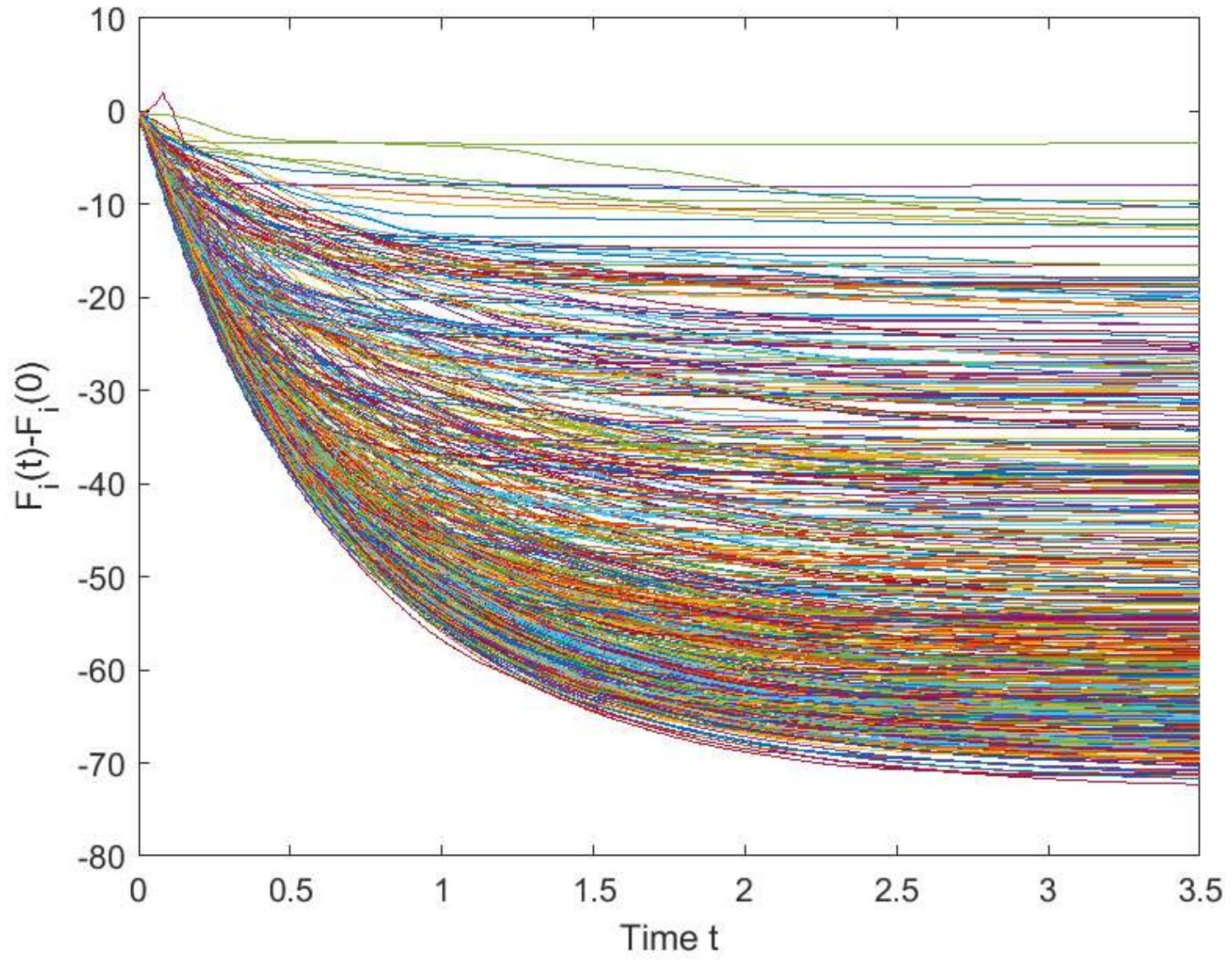}
        \captionsetup{font=small}
        \caption{The energy $\mathcal{F}_i(p_t)-\mathcal{F}_i(p_0)$ of 500 trajectories.}
        \label{figtraj1}
    \end{minipage}
    \hfill  
    \begin{minipage}[b]{0.48\textwidth}
        \centering
        \includegraphics[width=\textwidth]{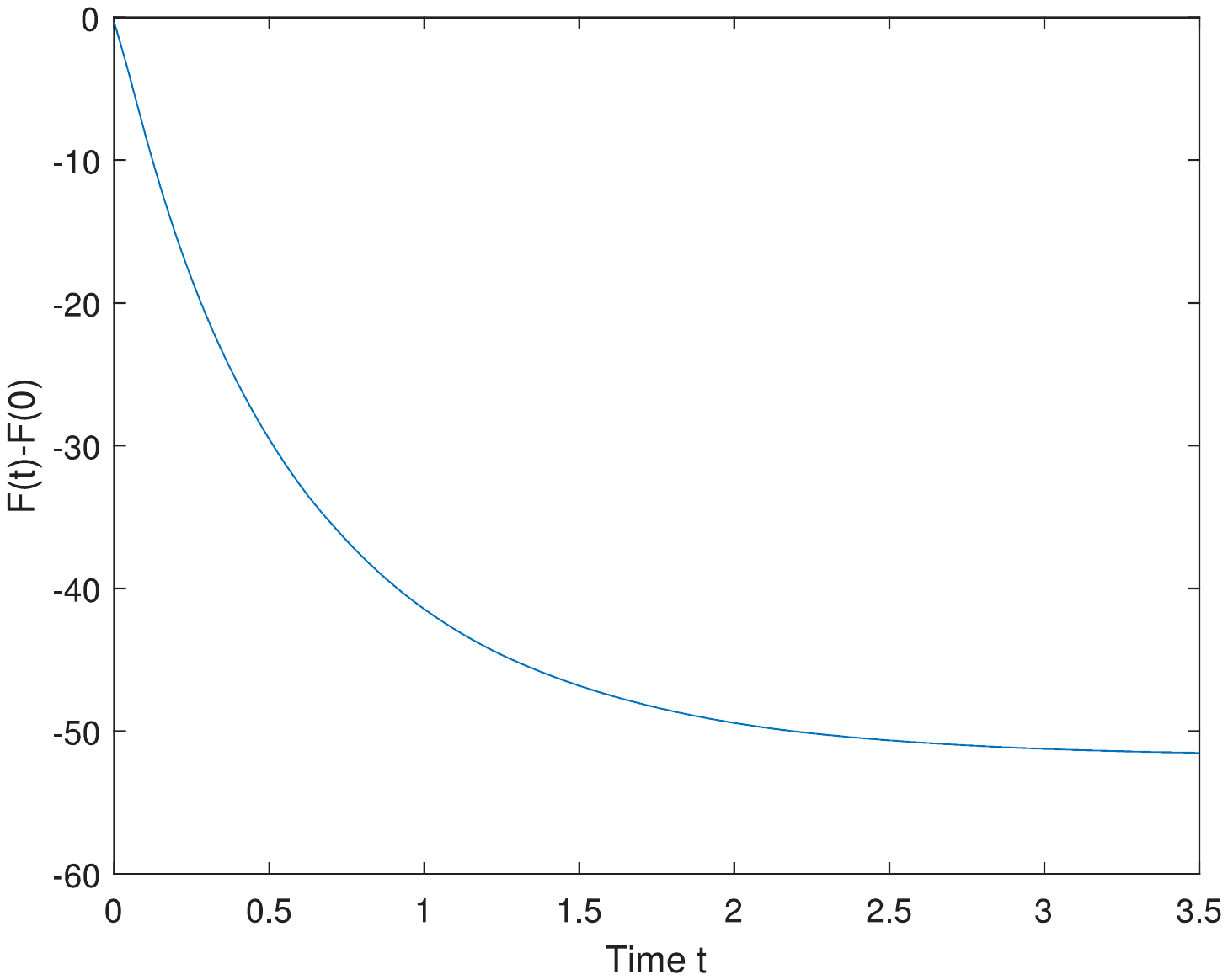}
        \captionsetup{font=small}
        \caption{The energy $\overline{\mathcal{F}}(p_t)-\overline{\mathcal{F}}(p_0)$ of \eqref{FD}.}
        \label{figenergy1}
    \end{minipage}
\end{figure}

\section{Further Discussions}\label{discussion}

In this section, we will make use of \eqref{560} for numerical simulations and observe the energy dissipation of the trajectories.
While the conditions for applying \eqref{560} are not strictly satisfied, these simulations may still provide valuable insights into the dynamical behavior of SDEs.
Based on these observations, we will pose two questions for further investigation.

\subsection{Condensation phenomena}
In this subsection we consider the case of $b(p)=1+p^\gamma$, where $\gamma \geq 1$. Let $\Phi(x)=\frac{|x|^2}{2}$,
we have the McKean-Vlasov SDE
\begin{equation}
\begin{aligned}
\rmd X(t)&= - X(t)\big(1+\rho(X(t))\big)^\gamma \rmd t + \sqrt{2} \rmd W(t), t\in[0,T], \\
X(0)&=\xi_0
\end{aligned}
\end{equation}
and the corresponding nonlinear Fokker-Planck equation
\begin{equation}\label{cp}
\begin{aligned}
\partial_t p(t,x) &= {\rm div} \big(x p(t,x)(1+ p(t,x)^\gamma) \big) + \Delta p(t,x), (t, x)\in [0,T] \times \R^d, \\
p(0,x)&=p_0(x), x\in \R^d.
\end{aligned}
\end{equation}
When $\gamma=1$, \eqref{cp} is also known as the Bose-Einstein-Fokker-Planck equation.

Since
$$g(s)=\frac{1}{\gamma}\big( \ln s^\gamma-\ln(1+s^\gamma) \big) +\frac{1}{\gamma} \ln 2$$
and
$$\eta(r) = \frac{1}{\gamma} \int_0^r \big( \ln s^\gamma-\ln(1+s^\gamma) \big) ds + \frac{r}{\gamma} \ln 2=:\frac{A(r)}{\gamma} + \frac{r}{\gamma} \ln 2, $$
we deduce the free energy functional
$$ \mathcal{F}(\rho)=\int_{\R^d} \Big( \frac{A(\rho)}{\gamma} + \frac{\rho}{\gamma} \ln 2 + \dfrac{|x|^2}{2}\rho \Big) dx$$
and the stationary solution is $p_{c, \infty}= \frac{1}{(c\rme^{\gamma \frac{|x|^2}{2}}-1)^{\frac{1}{\gamma}}}$.
Now choose $c_0$ such that $\int_{\R^d} p_{c_0, \infty} \rmd x =1$.
Similar to the Fermi-Dirac-Fokker-Planck equation, we select an initial density function $p_0$ such that both Assumption \ref{con} \ref{initial} and $p_{c_4, \infty}\leq p_0\leq p_{c_3, \infty}$ hold, where $0<c_3 <c_4$. The comparison principle implies $p_{c_4, \infty}\leq p(t,\cdot)\leq p_{c_3, \infty}$ for all $t \in [0,T]$ \cite{H}. It follows that Assumption \ref{con} \ref{b} is satisfied.
Applying \eqref{560} and \eqref{edi20} to \eqref{cp}, we compute the trajectorial rate of relative entropy dissipation
\begin{equation}\label{traj}
\begin{aligned}
& \lim_{t \downarrow t_0} \dfrac{\mathbb{E}_{\mathbb{P}}[\theta(t, X(t))| \mathcal{G}(t_0)] -\theta(t_0, X(t_0)) }{t-t_0}
= D(t_0, X(t_0)) \\
= & \Big( \ln \big(\rho(X(t_0)) \big)^\gamma -\ln\big(1+ \big(\rho(X(t_0))\big)^\gamma\big) \Big) \Big( \dfrac{2 \Delta \rho(X(t_0))}{\gamma \rho(X(t_0))} +\dfrac{1+\big( \rho(X(t_0)) \big)^\gamma}{\gamma} \\
& + \big \langle X(t_0), \nabla \rho(X(t_0)) \big \rangle \big(\rho(X(t_0)) \big)^{\gamma-1} -\dfrac{2\big(\nabla \rho(X(t_0))\big)^2}{\gamma \big(\rho(X(t_0)) \big)^2} \Big) \\
& + A\Big(\rho(X(t_0))\Big)\Big( -\dfrac{2\Delta \rho(X(t_0))}{\gamma \big(\rho(X(t_0)) \big)^2} -\dfrac{1+\big(\rho(X(t_0)) \big)^\gamma}{\gamma \rho(X(t_0))} \\
& -\big \langle X(t_0), \nabla \rho(X(t_0)) \big \rangle \big(\rho(X(t_0)) \big)^{\gamma-2}  +\dfrac{2\big (\nabla \rho(X(t_0))\big )^2}{\gamma \big( \rho(X(t_0)) \big)^3} \Big) \\
& +\dfrac{\big (\nabla \rho(X(t_0))\big)^2}{\big( \rho(X(t_0))\big) ^2} -\dfrac{\big( \rho(X(t_0))\big) ^{\gamma-2} \big(\nabla \rho(X(t_0))\big)^2}{1+\big(\rho(X(t_0)) \big)^\gamma}+1-X(t_0)^2\big(1+\big(\rho(X(t_0)) \big)^\gamma\big)
\end{aligned}
\end{equation}
in $L^1(\mathbb{P})$ and compute the energy dissipation
\begin{equation*}
\begin{aligned}
 \dfrac{\rmd }{\rmd t}\mathcal{F}(p(t,x))=- \mathcal{I}(p(t,x))
= -\int_{\R^d} \bigg| \dfrac{\nabla p(t,x)}{p(t,x)\big(1+p(t,x)^{\gamma}\big)}+x \bigg|^2\big(1+p(t,x)^{\gamma}\big)p(t,x) \rmd x,
\end{aligned}
\end{equation*}
where $\theta(t, x)= \frac{A(p(t,x))}{\gamma p(t,x)}+\frac{\ln 2}{\gamma}+\frac{|x|^2}{2}$.

According to \eqref{traj}, we perform numerical simulations for the trajectorial energy dissipation $\mathcal{F}_i(p_t)-\mathcal{F}_i(p_0)$ and its average $\overline{\mathcal{F}}(p_t)-\overline{\mathcal{F}}(p_0)$, where $i=1,...,500$.
In these simulations, the initial distribution is set as a Gaussian distribution with the mean of 20 and variance 1. Specifically, Fig.~\ref{figtraj2} and Fig.~\ref{figenergy2} illustrate the energy dissipation for $\gamma=1$, while Fig.~\ref{figtraj4} and Fig.~\ref{figenergy4} depict the energy dissipation for $\gamma=3$.

\begin{figure}[ht]
    \centering
    \begin{minipage}[b]{0.48\textwidth}
        \centering
        \includegraphics[width=\textwidth]{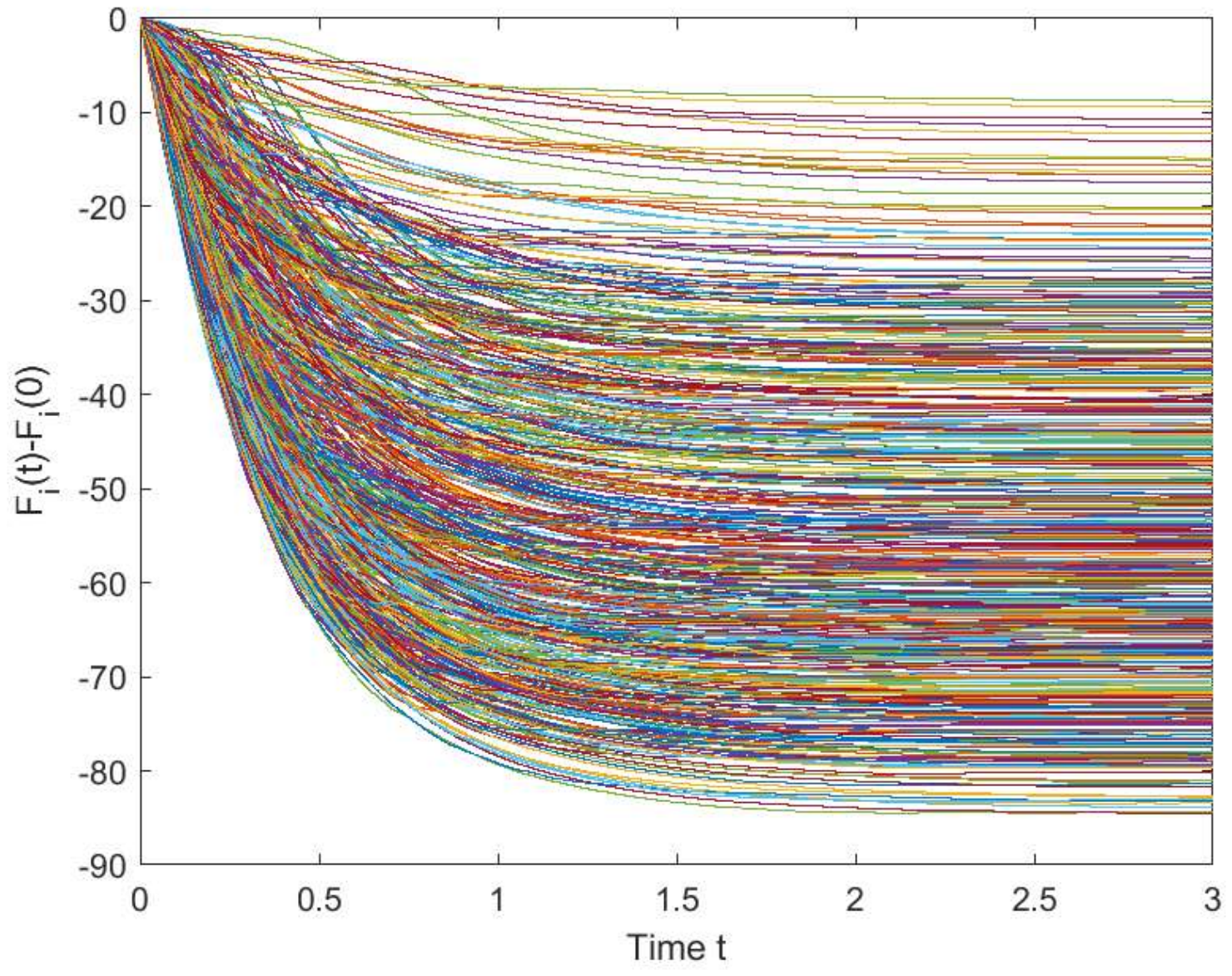}
        \captionsetup{font=small}
        \caption{For $\gamma=1$, the energy $\mathcal{F}_i(p_t)-\mathcal{F}_i(p_0)$ of 500 trajectories.}
        \label{figtraj2}
    \end{minipage}
    \hfill  
    \begin{minipage}[b]{0.48\textwidth}
        \centering
        \includegraphics[width=\textwidth]{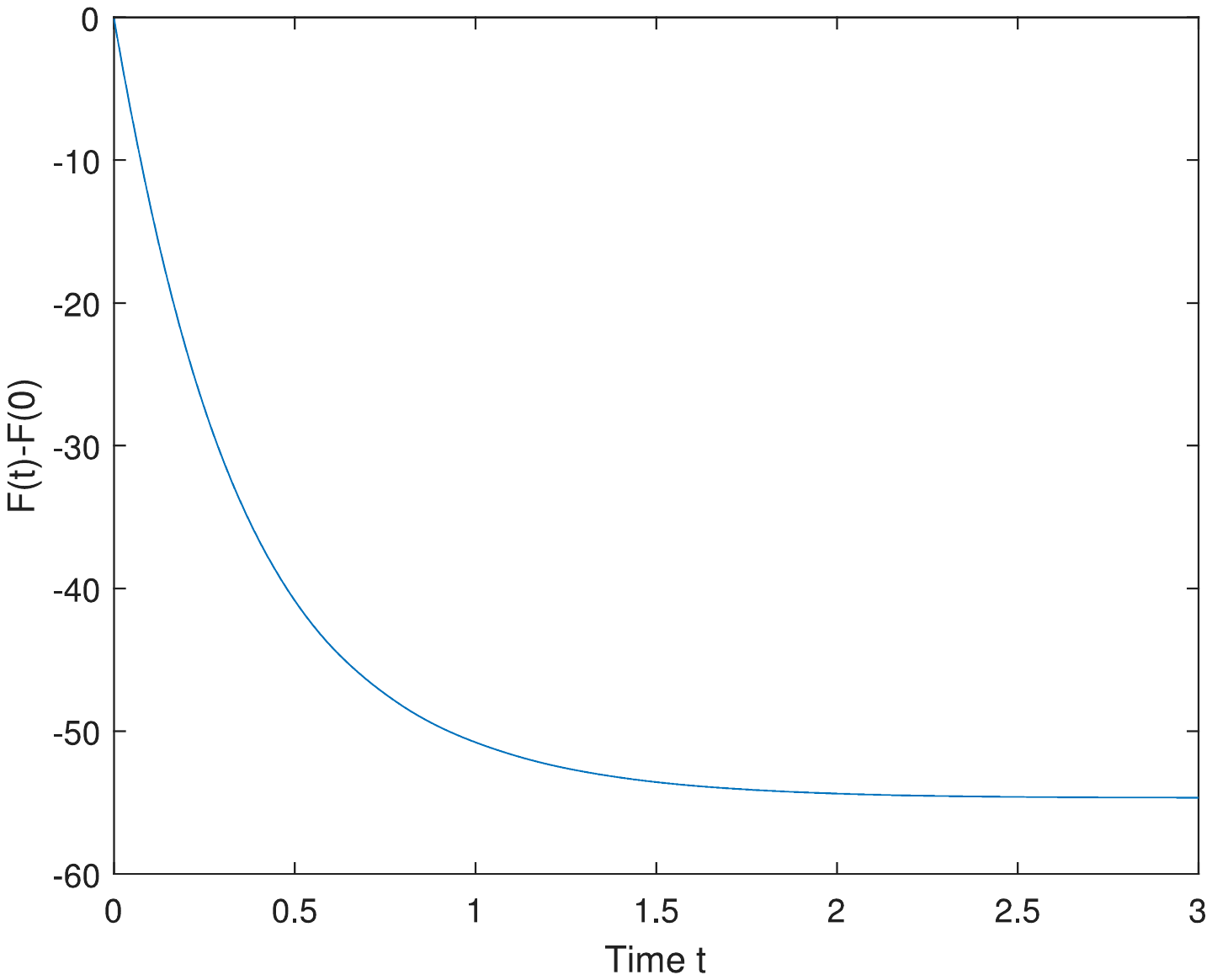}
        \captionsetup{font=small}
        \caption{For $\gamma=1$, the energy $\overline{\mathcal{F}}(p_t)-\overline{\mathcal{F}}(p_0)$ of \eqref{cp}.}
        \label{figenergy2}
    \end{minipage}
\end{figure}

\begin{figure}[ht]
    \centering
    \begin{minipage}[b]{0.48\textwidth}
        \centering
        \includegraphics[width=\textwidth]{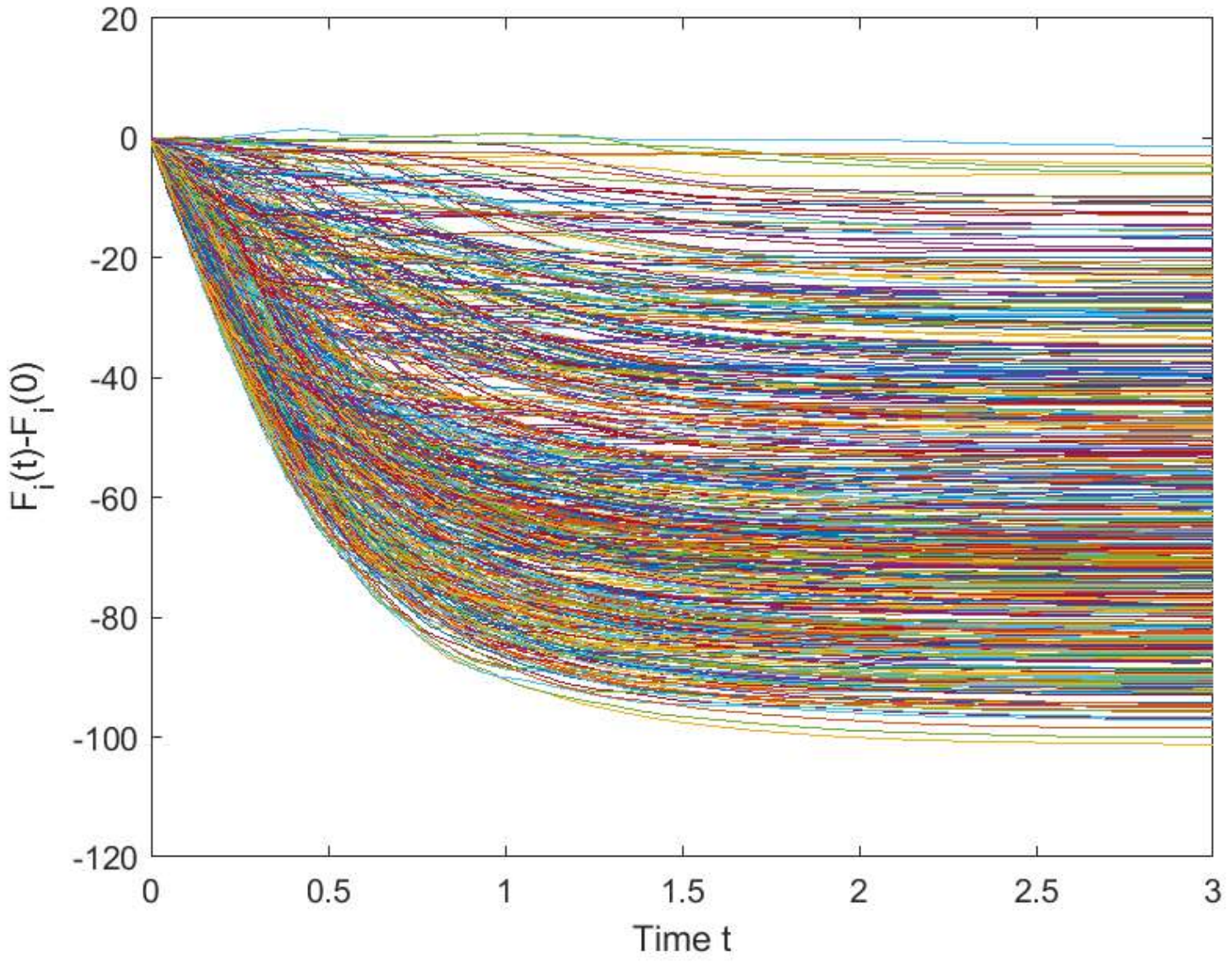}
        \captionsetup{font=small}
        \caption{For $\gamma=3$, the energy $\mathcal{F}_i(p_t)-\mathcal{F}_i(p_0)$ of 500 trajectories.}
        \label{figtraj4}
    \end{minipage}
    \hfill  
    \begin{minipage}[b]{0.48\textwidth}
        \centering
        \includegraphics[width=\textwidth]{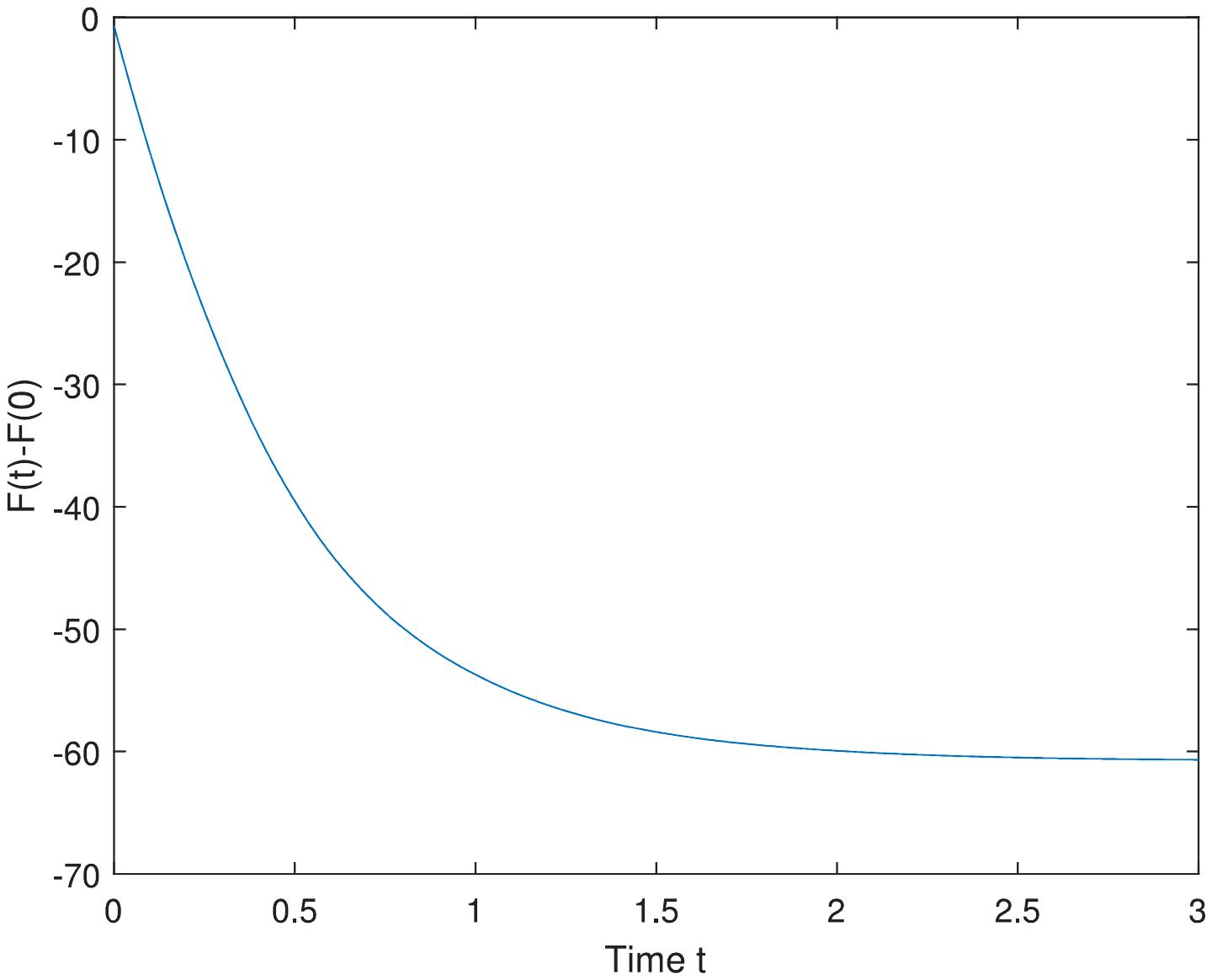}
        \captionsetup{font=small}
        \caption{For $\gamma=3$, the energy $\overline{\mathcal{F}}(p_t)-\overline{\mathcal{F}}(p_0)$ of \eqref{cp}.}
        \label{figenergy4}
    \end{minipage}
\end{figure}

It was pointed out in \cite{CHR, T} that the formation of condensates occurs when $d>2$ and the initial mass $m > m_c$,  where $m_c=\int_{\R^d} p_{1, \infty} \rmd x$ denotes the critical mass.
Moreover, $m_c>1$ when $\gamma>1$. The long-time behavior of the system when $m>m_c$ remains an open problem.
Generally, the distribution of solutions to SDEs is a probability measure, i.e. $m=1$, which presents challenges in directly describing condensation phenomena.
Nevertheless, we believe that SDEs could offer a promising framework for investigating this issue in future research. Based on this observation, we propose the following

\textbf{Question 1:} Under what conditions does the scenario $m_c<1$ occur, and what are the implications of the condensation phenomenon for the evolution of solutions to SDEs in this case?

\subsection{Non-exponential convergence}

We now focus on the case where $b(p)=p^\alpha$ and $\Phi(x)=\frac{|x|^2}{2}$, with $\alpha \geq 1$.
In this case, the McKean-Vlasov SDE is expressed as
\begin{equation}
\begin{aligned}
\rmd X(t)&= - X(t)\big(\rho(X(t))\big)^{\alpha -1} \rmd t + \sqrt{2} \rmd W(t), t\in[0, T],  \\
X(0)&=\xi_0
\end{aligned}
\end{equation}
and the corresponding nonlinear Fokker-Planck equation takes the form
\begin{equation}\label{pc}
\begin{aligned}
\partial_t p(t,x)& = {\rm div} \big(x p(t,x)^\alpha\big) + \Delta p(t,x), (t, x)\in [0, T]\times\R^d, \\
p(0,x)&=p_0(x), x\in \R^d.
\end{aligned}
\end{equation}

\textbf{Case 1: $\alpha=1$.} Similar to Remark \ref{prop.energy} (iii), we adjust the integration interval of $\eta(r)$ to avoid discontinuities, which does not impact the dissipation of energy.
We redefine $\eta(r):=\int_1^r g(s)\rmd s$ and $\mathcal{F}(\rho) :=\int_{\R^d} \big(\eta (\rho) +\Phi(x) \rho +1 \big) \rmd x$.
Consequently
$$g(s)=-\dfrac{1}{s}+1, \eta(r) = -\ln r+r-1$$
and
$$ \mathcal{F}(\rho)=\int_{\R^d} \Big( -\ln \rho +\rho + \dfrac{|x|^2}{2}\rho \Big) \rmd x.$$

According to \eqref{560} and \eqref{edi20}, we calculate the trajectorial rate of relative entropy dissipation
\begin{equation*}
\begin{aligned}
& \lim_{t \rightarrow t_0} \dfrac{\mathbb{E}_{\mathbb{P}}[\theta(t, X(t))| \mathcal{G}(t_0)] -\theta(t_0, X(t_0)) }{t-t_0}
= D(t_0, X(t_0)) \\
= & \ln \rho(X(t_0)) + 2\big\langle X(t_0), \nabla \rho(X(t_0)) \big\rangle \dfrac{\ln \rho(X(t_0))}{\rho(X(t_0))} + \Delta \rho(X(t_0)) \dfrac{\ln \rho(X(t_0))}{\big(\rho(X(t_0))\big) ^2}\\
& -\dfrac{\Delta \rho(X(t_0))}{\rho(X(t_0))} +\dfrac{\big(\nabla \rho(X(t_0))\big)^2}{\big( \rho(X(t_0)) \big)^2} +\Delta \rho(X(t_0)) +1 + \big\langle X(t_0), \nabla \rho(X(t_0)) \big\rangle \\
&-\big\langle X(t_0), \nabla \rho(X(t_0)) \big\rangle \rho(X(t_0)) -X(t_0)^2 \rho(X(t_0))
\end{aligned}
\end{equation*}
in $L^1(\mathbb{P})$ and calculate the energy dissipation
\begin{equation*}
\begin{aligned}
& \dfrac{\rmd }{\rmd t}\mathcal{F}(p(t,x))=-\mathcal{I}(p(t,x))
= -\int_{\R^d} \bigg| \dfrac{\nabla p(t,x)}{p(t,x)^2}+x \bigg|^2 p(t,x)^2 \rmd x,
\end{aligned}
\end{equation*}
where $\theta(t, x)= -\frac{\ln p(t,x)}{p(t,x)}+1-\frac{1}{p(t,x)} +\frac{|x|^2}{2} $.

\textbf{Case 2: $\alpha>1$.} For $\alpha>1$, we obtain
$$g(s)=-\dfrac{s^{-\alpha}-1}{\alpha}, \eta(r) = -\dfrac{r^{1-\alpha}}{\alpha(1-\alpha)}+\dfrac{r}{\alpha}$$
and
$$ \mathcal{F}(\rho)=\int_{\R^d} \Big( -\dfrac{\rho^{1-\alpha}}{\alpha(1-\alpha)}+\dfrac{\rho}{\alpha} + \dfrac{|x|^2}{2}\rho \Big) \rmd x.$$

In view of \eqref{560} and \eqref{edi20}, we have the trajectorial rate of relative entropy dissipation
\begin{equation*}
\begin{aligned}
& \lim_{t \rightarrow t_0} \dfrac{\mathbb{E}_{\mathbb{P}}[\theta(t, X(t))| \mathcal{G}(t_0)] -\theta(t_0, X(t_0)) }{t-t_0}
= D(t_0, X(t_0)) \\
= & \dfrac{1}{1-\alpha} +\dfrac{(\alpha+1)\big\langle X(t_0), \nabla \rho(X(t_0)) \big\rangle}{(1-\alpha )\rho(X(t_0))} +\dfrac{\Delta \rho(X(t_0))}{(1-\alpha) \big( \rho(X(t_0)) \big)^{1+\alpha}} +\dfrac{\big(\nabla \rho(X(t_0))\big)^2}{\big(\rho(X(t_0)) \big)^{1+\alpha}}  \\
&-\dfrac{\Delta \rho(X(t_0))}{\alpha \big(\rho(X(t_0)) \big)^{\alpha}} +\dfrac{\Delta \rho(X(t_0))}{\alpha} +1 +\dfrac{\big\langle X(t_0), \nabla \rho(X(t_0)) \big\rangle}{\alpha} -X(t_0)^2 \big( \rho(X(t_0)) \big)^{\alpha}\\
& -\dfrac{\big\langle X(t_0), \nabla \rho(X(t_0)) \big\rangle}{\alpha} \big(\rho(X(t_0)) \big)^{\alpha}
\end{aligned}
\end{equation*}
in $L^1(\mathbb{P})$ and the energy dissipation
\begin{equation*}
\begin{aligned}
 \dfrac{\rmd }{\rmd t}\mathcal{F}(p(t,x))=-\mathcal{I}(p(t,x))
= -\int_{\R^d} \bigg| \dfrac{\nabla p(t,x)}{p(t,x)^{\alpha+1}}+x \bigg|^2 p(t,x)^{\alpha+1} \rmd x,
\end{aligned}
\end{equation*}
where $\theta(t, x)= -\frac{( p(t,x) )^{1-\alpha}}{\alpha(1-\alpha)} +\frac{p(t,x)}{\alpha} +\frac{|x|^2}{2} $.

Now, let the initial density function be $p_0(x)=\frac{1}{\sqrt{2\pi}}\rme^{-\frac{(x-20)^2}{2}}$, and consider 500 trajectories.
Then, Fig.~\ref{figtraj5}-Fig.~\ref{figenergy6} illustrate the energy dissipation for $\alpha=1$ and $\alpha=2$.
Numerical results show that for $\alpha=1$, the energy dissipation converges at an exponential rate.
However, for $\alpha=2$, the convergence rate is less evident, suggesting the possibility of non-exponential convergence, such as polynomial convergence.
We hypothesize that this change arises because the condition $b\geq b_0$ for some constant $b_0>0$ is not satisfied.
This significant shift introduces a new question, which is worthy of further investigation in our eyes.

\textbf{Question 2:} If $b$ is not bounded from below, is it possible to identify an appropriate free energy functional and metric such that \eqref{pc} admits a gradient flow representation?
Furthermore, how can the rate of energy dissipation be accurately characterized in this case?

\begin{figure}[ht]
    \centering
    \begin{minipage}[b]{0.48\textwidth}
        \centering
        \includegraphics[width=\textwidth]{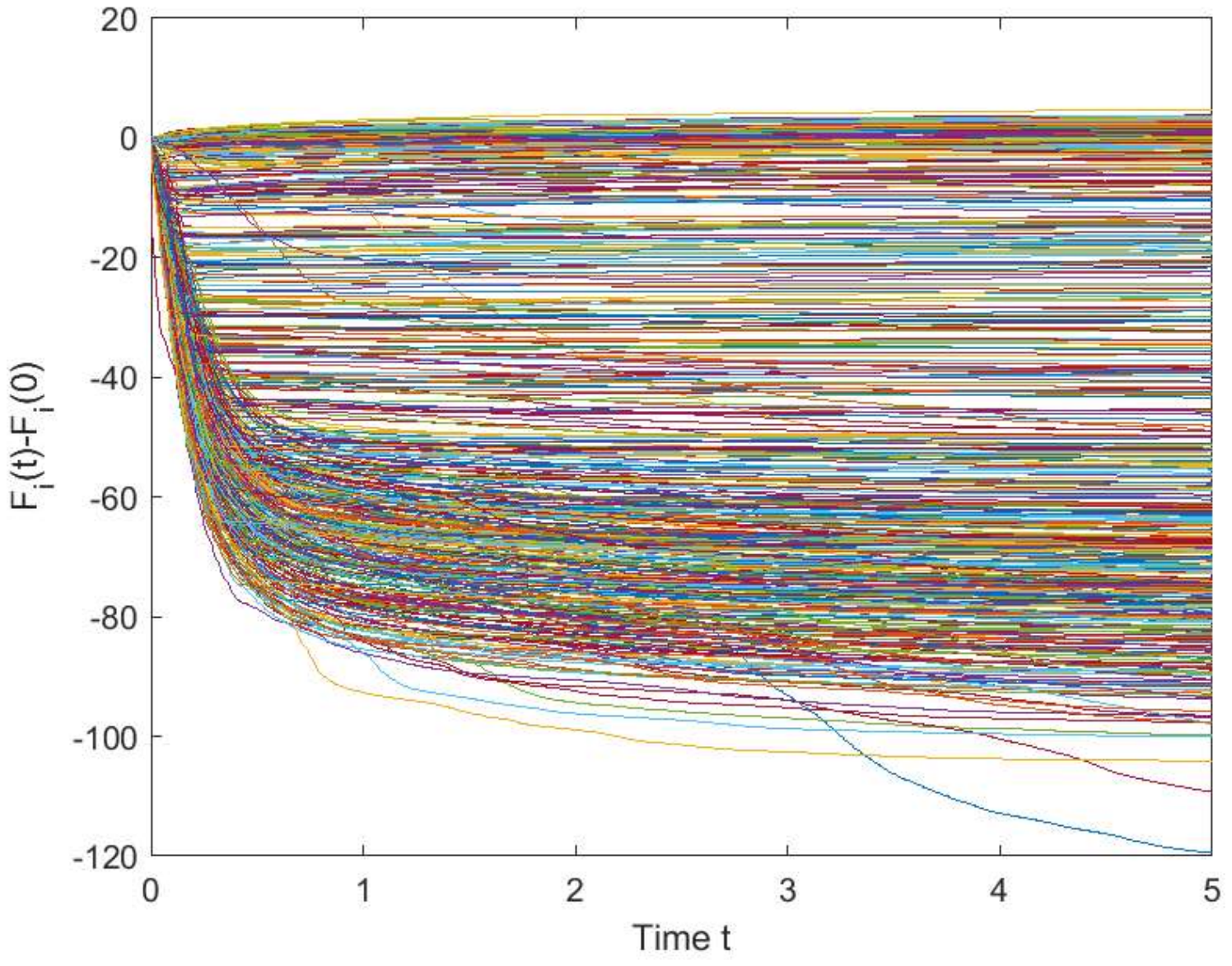}
        \captionsetup{font=small}
        \caption{For $\alpha=1$, the energy $\mathcal{F}_i(p_t)-\mathcal{F}_i(p_0)$ of 500 trajectories.}
        \label{figtraj5}
    \end{minipage}
    \hfill  
    \begin{minipage}[b]{0.48\textwidth}
        \centering
        \includegraphics[width=\textwidth]{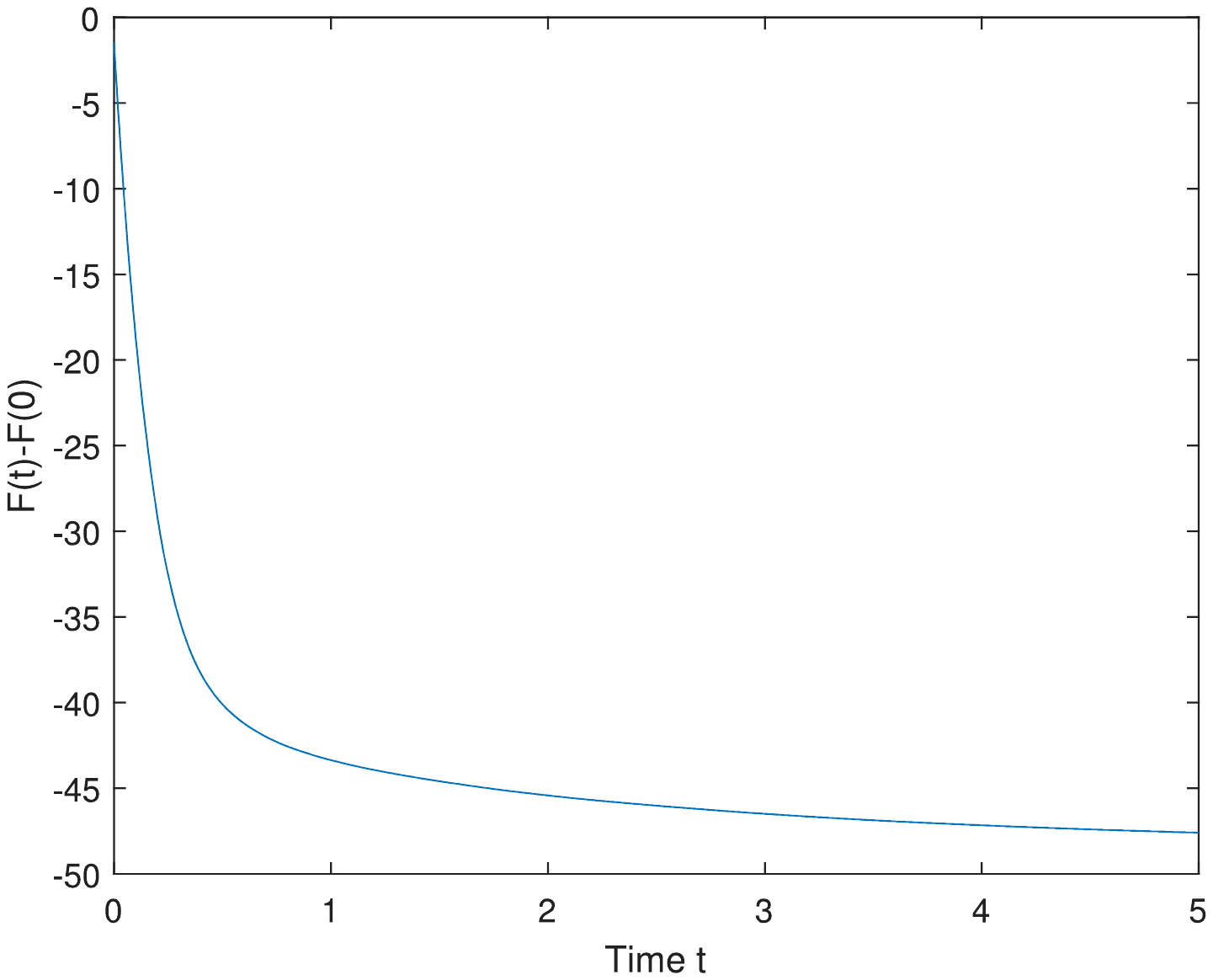}
        \captionsetup{font=small}
        \caption{For $\alpha=1$, the energy $\overline{\mathcal{F}}(p_t)-\overline{\mathcal{F}}(p_0)$ of \eqref{pc}.}
        \label{figenergy5}
    \end{minipage}
\end{figure}

\begin{figure}[ht]
    \centering
    \begin{minipage}[b]{0.48\textwidth}
        \centering
        \includegraphics[width=\textwidth]{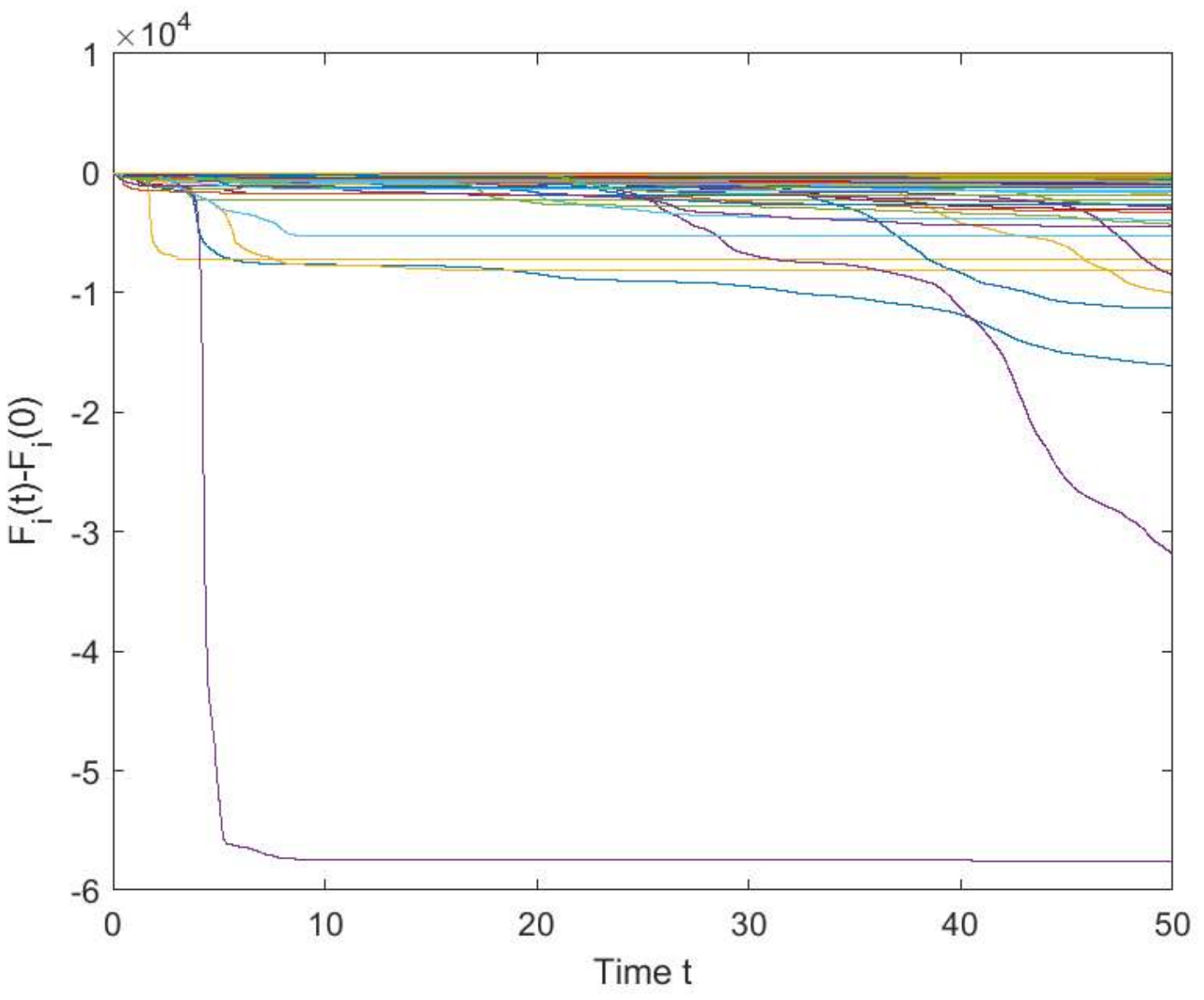}
        \captionsetup{font=small}
        \caption{For $\alpha=2$, the energy $\mathcal{F}_i(p_t)-\mathcal{F}_i(p_0)$ of 500 trajectories.}
        \label{figtraj6}
    \end{minipage}
    \hfill  
    \begin{minipage}[b]{0.48\textwidth}
        \centering
        \includegraphics[width=\textwidth]{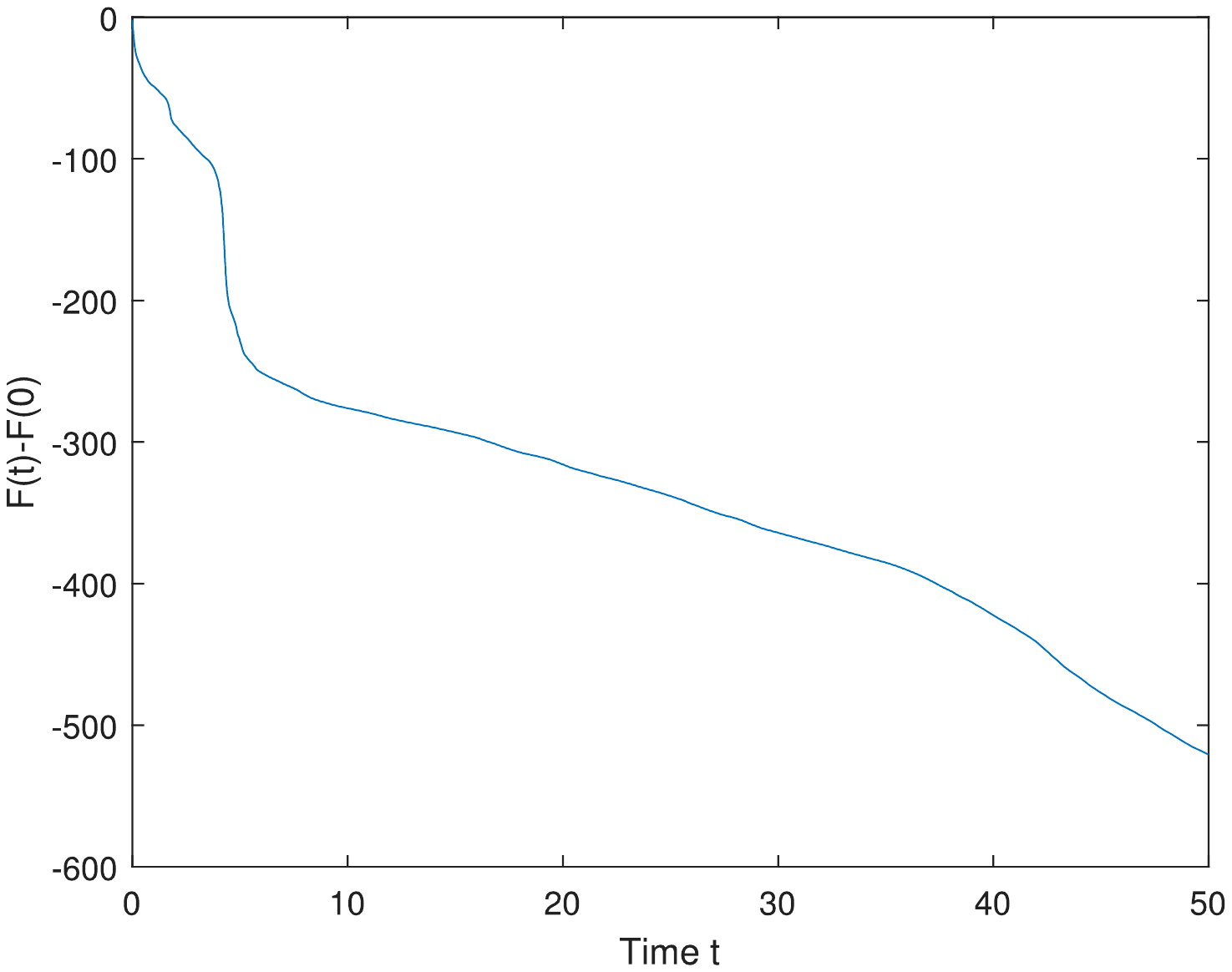}
        \captionsetup{font=small}
        \caption{For $\alpha=2$, the energy $\overline{\mathcal{F}}(p_t)-\overline{\mathcal{F}}(p_0)$ of \eqref{pc}.}
        \label{figenergy6}
    \end{minipage}
\end{figure}

\appendix
\section{Proof of Lemma \ref{integrability}}\label{appendix}

Before formally beginning the proof, we mention some conventions.
We denote convolution by
\begin{equation}\label{conv1}
p \ast \omega (t,x):=\int_{\R^d} \omega(x-y) p(t, y) \rmd y
\end{equation}
for $\omega(x) \in C_0^{\infty}(\R^d)$.
The convolution employed in the remainder of the proof are all with respect to the variable $x \in \R^d$.
Let $\sigma(x)=\frac{1}{\sqrt{2\pi}}e^{-\frac{x^2}{2}}$ represent the standard Gaussian density and $\omega_{\varepsilon}(x)= \varepsilon^{-d} \sigma(\frac{x}{\varepsilon})$ for $\varepsilon \in (0, 1)$. For convenience, we will abbreviate $\partial_{x_i x_j} f(x)$ as $\partial_{ij}f(x)$ and adopt the Einstein summation convention moving forward.

Following a similar approach to the proof in \cite[Theorem 2.1]{BRS}, we will proceed based on the following result \cite[Lemma 2.1]{BKR}.
Given two nonnegative functions $f_1, f_2 \in L^1(\R^d)$, for any measurable function $\psi$ such that $|\psi|^2 f_1 \in L^1(\R^d)$,
\begin{equation}\label{inequality}
\int_{\R^d} \dfrac{|(\psi f_1)\ast f_2|^2}{f_1 \ast f_2} \rmd x \leq \int_{\R^d} |\psi|^2 f_1 \rmd x \int_{\R^d} f_2 \rmd x.
\end{equation}
Here we follow the convention that $\frac{|(\psi f_1)\ast f_2(x)|^2}{f_1 \ast f_2(x)}:=0$ if $f_1 \ast f_2(x)=0$.

According to Assumption \ref{con} and \eqref{NFPE}, we obtain
\begin{equation}\label{convolution}
\begin{aligned}
\partial_t(p \ast \omega) & =\int_{\R^d} [\omega(x-y)\partial_{ii}f(p(t, y))+\omega(x-y) \partial_i(\nabla\Phi^i(y)b(p(t,y)) p(t,y))]\rmd y \\
& = \int_{\R^d} [\partial_{y_i y_i}\omega(x-y) f(p(t, y)) -\partial_{y_i} \omega(x-y)\nabla\Phi^i(y) b(p(t,y)) p(t,y)] \rmd y \\
& = \int_{\R^d} [\partial_{ii} \omega(x-y) f(p(t,y)) +\partial_i \omega(x-y)\nabla\Phi^i(y) b(p(t,y))p(t,y)] \rmd y\\
& = f(p) \ast \partial_{ii}\omega (t, x) +[\nabla \Phi^i b(p) p] \ast \partial_i \omega (t, x).
\end{aligned}
\end{equation}
By integrating \eqref{convolution} with respect to time from 0 to $t \in [0, T]$, we obtain
\begin{equation}\label{conv2}
\begin{aligned}
p\ast \omega(t,x) & =p\ast \omega(0,x) +\int_0^t \upsilon(s, p(s, x), x) \rmd s \\
& =p\ast \omega(0,x) +\int_0^t [f(p) \ast \partial_{ii}\omega(t, x) +[\nabla \Phi^i b(p) p] \ast \partial_i \omega(t, x)] \rmd s.
\end{aligned}
\end{equation}

To make the proof rigorous, we will show that \eqref{conv2} is consistent with the initial definition \eqref{conv1} for a.e $t\in [0, T]$ and $x\in \R^d$.
Since Assumptions \ref{con} and $\omega(x) \in C_0^{\infty}(\R^d)$ hold, it follows readily that $\upsilon(s, p(s, x), x) \in L^1([0,T]\times \R^d)$ and $p \ast \omega(0,x)$ is bounded.
Furthermore, function $p\ast \omega(t,x)$ is absolutely continuous with respect to $t\in [0, T]$ and smooth with respect to $x\in \R^d$. Combined with $\|p\ast \omega(t,x)\|_{\infty} \leq \|\omega(x)\|_{\infty}$ for all $t \in [0, T]$, it follows that $p\ast \omega(t,x)\in C_b^{\infty}([0,T] \times \R^d)$.

The above analysis are also applicable for $\omega(x)=\omega_{\varepsilon}(x)$.
We define $p_{\varepsilon}=p\ast \omega_{\varepsilon}$, $\Lambda(x)=\ln \max (|x|, 1)$
and $h_{\varepsilon}(t, x)=p_{\varepsilon}(t, x)+\varepsilon \max (|x|, 1)^{-d-1}$.
Assumption \ref{con} guarantees that $\Lambda \in L^1(p(t, \cdot))$ for $t\in [0, T]$. Taking $t_0\in (0,T]$, it is not hard to get $\int_{\R^d} p(t_0, x)\Lambda(x) \rmd x < +\infty$.
We choose $\varepsilon=\frac{1}{n}$, where $n$ is any positive integer.
According to $\Lambda(x+y) \leq \Lambda(x)+|y|$, which can be verified using a classification-based approach, it follows that
\begin{equation}\label{11}
\begin{aligned}
\int_{\R^d} h_{\varepsilon}(t_0, x) \Lambda(x)\rmd x
& \leq \int_{\R^d} p_{\varepsilon}(t_0, x) \Lambda (x)\rmd x + \int_{\R^d} \varepsilon \max (|x|, 1)^{-d-1} \rmd x \\
& \leq \int_{\R^d} p(t_0, x)\Lambda (x)\rmd x +\int_{\R^d} |y|\omega_{\varepsilon}(y)\rmd y +\varepsilon\int_{\R^d}\max (|x|, 1)^{-d-1}\Lambda(x) \rmd x \\
& \leq \int_{\R^d} p(t_0, x)\Lambda (x)\rmd x + \varepsilon M_1,
\end{aligned}
\end{equation}
where $M_1$ is independent of the value of $\varepsilon$.

Next we want to prove that
\begin{equation}\label{12}
\begin{aligned}
& \int_0^{t_0} \!\! \int_{\R^d} \partial_t(p\ast \omega_{\varepsilon})(t,x) \ln h_{\varepsilon}(t,x) \rmd x \rmd t \\
=& \int_0^{t_0} \!\! \int_{\R^d} [f(p) \ast \partial_{ii}\omega_{\varepsilon} (t, x) +[\nabla \Phi^i b(p) p] \ast \partial_i \omega_{\varepsilon} (t, x)] \ln h_{\varepsilon}(t,x) \rmd x \rmd t.
\end{aligned}
\end{equation}
This follows from \eqref{convolution} as well as the proof of integrability at time $t\in [0, t_0]$.
For $|\ln h_{\varepsilon}|$, we can obtain
\begin{equation*}
|\ln h_{\varepsilon}| =|\ln(p_{\varepsilon}+\varepsilon\max(|x|, 1)^{-d-1})|
\leq C_1+C_2\Lambda
\end{equation*}
for some constants $C_1$ and $C_2$.
The inequality \eqref{inequality} and Assumption \ref{con} can guarantee that
\begin{equation}\label{13}
\begin{aligned}
\|([\nabla \Phi^i b(p)p]\ast \partial_i \omega_{\varepsilon})(p\ast |\partial_i \omega_{\varepsilon}|)^{-\frac{1}{2}}\|_{L^2}^2
\leq \int_{\R^d} |\nabla \Phi^i b(p) p|^2 p\rmd x \cdot \int_{\R^d}\partial_i \omega_{\varepsilon} \rmd x < +\infty.
\end{aligned}
\end{equation}
It can be easily concluded through a classification-based argument that
\begin{equation*}
\Lambda(x+y)^2 \leq 4 +2\Lambda(x)^2 +2\Lambda(y)^2.
\end{equation*}
Then $\Lambda(x)^2 \leq 4 +2\Lambda(x-y)^2 +2\Lambda(y)^2$ and $\Lambda(x) \in L^2(p(t, \cdot))$ together guarantee that
\begin{equation}\label{14}
\begin{aligned}
&\|(p\ast |\partial_i\omega_{\varepsilon}|)^{\frac{1}{2}} \Lambda\|_{L^2}^2 \\
\leq & \int_{\R^d} 4 p\ast |\partial_i \omega_{\varepsilon}|\rmd x +2\int_{\R^d} |\partial_i \omega_{\varepsilon}(y)|\Big(\int_{\R^d}p(t, x-y)\Lambda(x-y)^2 \rmd x\Big) \rmd y \\
& + 2\int_{\R^d}\Big(\int_{\R^d}|\partial_i\omega_{\varepsilon}(x-y)|\rmd x\Big)p(t, y)\Lambda(y)^2 \rmd y
< +\infty.
\end{aligned}
\end{equation}
Combining \eqref{13}, \eqref{14} and H\"older's inequality, we obtain
\begin{equation*}\label{15}
\begin{aligned}
& \|([\nabla \Phi^i b(p) p] \ast \partial_i \omega_{\varepsilon})\Lambda\|_{L^1} \\
= &\|([\nabla \Phi^i b(p)p]\ast \partial_i \omega_{\varepsilon})(p\ast |\partial_i \omega_{\varepsilon}|)^{-\frac{1}{2}} \cdot (p\ast |\partial_i\omega_{\varepsilon}|)^{\frac{1}{2}} \Lambda \|_{L^1} \\
\leq & \|([\nabla \Phi^i b(p)p]\ast \partial_i \omega_{\varepsilon})(p\ast |\partial_i \omega_{\varepsilon}|)^{-\frac{1}{2}}\|_{L^2}^2 \cdot \|(p\ast |\partial_i\omega_{\varepsilon}|)^{\frac{1}{2}} \Lambda\|_{L^2}^2
 < +\infty,
\end{aligned}
\end{equation*}
which means $([\nabla \Phi^i b(p) p] \ast \partial_i \omega_{\varepsilon} )\Lambda$ is integrable.
The inequality \eqref{inequality} and Assumption \ref{con} can guarantee that
\begin{equation}\label{16}
\begin{aligned}
\|(f(p) \ast \partial_{ii}\omega_{\varepsilon}) (p\ast |\partial_{ii} \omega_{\varepsilon}|)^{-\frac{1}{2}}\|_{L^2}^2
\leq \int_{\R^d}\Big|\dfrac{f(p)}{p}\Big|^2 p \rmd x \cdot\int_{\R^d} \partial_{ii} \omega_{\varepsilon} \rmd x
< +\infty.
\end{aligned}
\end{equation}
Similar to \eqref{14}, we obtain
\begin{equation}\label{17}
\|(p\ast |\partial_{ii}\omega_{\varepsilon}|)^{\frac{1}{2}} \Lambda\|_{L^2}^2 <+\infty.
\end{equation}
Combining \eqref{16}, \eqref{17} and H\"older's inequality, we obtain
\begin{equation*}\label{18}
\begin{aligned}
& \|(f(p) \ast \partial_{ii}\omega_{\varepsilon})\Lambda\|_{L^1}
 \leq \|(f(p) \ast \partial_{ii}\omega_{\varepsilon}) (p\ast |\partial_{ii} \omega_{\varepsilon}|)^{-\frac{1}{2}}\|_{L^2}^2 \cdot
\|(p\ast |\partial_{ii}\omega_{\varepsilon}|)^{\frac{1}{2}} \Lambda\|_{L^2}^2
 <+\infty.
\end{aligned}
\end{equation*}
Therefore, we have
\begin{equation}\label{19}
\begin{aligned}
& \int_0^{t_0} \!\! \int_{\R^d} \partial_t p_{\varepsilon} \ln h_{\varepsilon}(t,x) \rmd x \rmd t \\
= & \int_0^{t_0} \!\! \int_{\R^d} \partial_i [f(p) \ast \partial_{i}\omega_{\varepsilon} (t, x) +[\nabla \Phi^i b(p) p] \ast \omega_{\varepsilon} (t, x)] \ln h_{\varepsilon}(t,x) \rmd x \rmd t \\
= & -\int_0^{t_0} \!\! \int_{\R^d}[f(p)\ast \partial_i \omega_{\varepsilon}(t,x) +(\nabla \Phi^i b(p) p)\ast \omega_{\varepsilon}(t, x)]\dfrac{\partial_i h_{\varepsilon}}{h_{\varepsilon}}(t,x) \rmd x \rmd t
\end{aligned}
\end{equation}
with the help of \eqref{12}. Similarly, integrability is a necessary condition for \eqref{19} to hold.
Since $p_{\varepsilon} \leq h_{\varepsilon}$, $b\leq b_1$ and $|\nabla \Phi(x)|\leq C|x|$ for all $|x|>R$, we get
\begin{equation}\label{20}
\begin{aligned}
\Big\|\dfrac{(\nabla \Phi^i b(p) p)\ast \omega_{\varepsilon}}{h_{\varepsilon}^{\frac{1}{2}}}\Big\|_{L^2}^2
 & \leq \int_0^{t_0} \!\! \int_{\R^d} \dfrac{[(\nabla \Phi^i b(p) p)\ast \omega_{\varepsilon}]^2}{p\ast \omega_{\varepsilon}} \rmd x \rmd t \\
& \leq \int_0^{t_0} \!\! \int_{\R^d} |\nabla \Phi^i b(p)|^2p \rmd x\rmd t
 < +\infty.
\end{aligned}
\end{equation}
According to \eqref{inequality}, $p_{\varepsilon} \leq h_{\varepsilon}$ and $ \varepsilon \max(|x|, 1)^{-d-1}\leq h_{\varepsilon}$, we have
\begin{equation}\label{21}
\begin{aligned}
& \bigg\|\dfrac{\nabla h_{\varepsilon}}{h_{\varepsilon}^{\frac{1}{2}}}\bigg\|_{L^2}^2
 \leq \int_0^{t_0} \!\! \int_{\R^d} \dfrac{2|\nabla p_{\varepsilon}|^2}{h_{\varepsilon}} +\dfrac{2\varepsilon^2(d+1)^2|x|^{-2d-4}1_{|x|\geq 1}}{h_{\varepsilon}}\rmd x \rmd t \\
\leq & 2\int_0^{t_0} \!\! \int_{\R^d} \dfrac{|p \ast \nabla \omega_{\varepsilon}|^2}{p\ast \omega_{\varepsilon}} \rmd x\rmd t +T2\varepsilon(d+1)^2\int_{\{|x|\geq 1\}} \dfrac{|x|^{-2d-4}}{|x|^{-d-1}}\rmd x \\
\leq & 2\int_0^{t_0}\Big( \int_{\R^d}p \rmd x \cdot \int_{\R^d}\dfrac{|\nabla\omega_{\varepsilon}|^2}{\omega_{\varepsilon}}\rmd x\Big )\rmd t +T2\varepsilon(d+1)^2\int_{\{|x|\geq 1\}} |x|^{-d-3}\rmd x
 < +\infty.
\end{aligned}
\end{equation}
Since \eqref{inequality} and $f'(r)\leq \gamma_1$, we obtain
\begin{equation}\label{22}
\begin{aligned}
\bigg\|\dfrac{f(p)\ast \partial_i\omega_{\varepsilon}}{h_{\varepsilon}^{\frac{1}{2}}}\bigg\|_{L^2}^2
 &\leq \int_0^{t_0} \!\! \int_{\R^d} \dfrac{|f(p)\ast \partial_i\omega_{\varepsilon}|^2}{p\ast \omega_{\varepsilon}} \rmd x \rmd t \\
 &\leq \gamma_1^2\int_0^{t_0}\Big(\int_{\R^d}p\rmd x \cdot \int_{\R^d} \dfrac{|\partial_i\omega_{\varepsilon}|^2}{\omega_{\varepsilon}}\rmd x\Big)\rmd t
< +\infty.
\end{aligned}
\end{equation}
Combining \eqref{20}, \eqref{21} and \eqref{22}, we get
\begin{equation*}\label{23}
\begin{aligned}
[f(p)\ast \partial_i \omega_{\varepsilon} +(\nabla \Phi^i b(p) p)\ast \omega_{\varepsilon}]
\dfrac{\partial_i h_{\varepsilon}}{h_{\varepsilon}}
=\Big(\dfrac{f(p)\ast \partial_i\omega_{\varepsilon}}{h_{\varepsilon}^{\frac{1}{2}}} +\dfrac{(\nabla \Phi^i b(p) p)\ast \omega_{\varepsilon}}{h_{\varepsilon}^{\frac{1}{2}}} \Big)\dfrac{\partial_i h_{\varepsilon}}{h_{\varepsilon}^{\frac{1}{2}}}
\in L^1(\R^d),
\end{aligned}
\end{equation*}
which implies that $\partial_t p_{\varepsilon} \ln h_{\varepsilon}(t,x)$ is integrable.
From
\begin{equation*}\label{24}
\begin{aligned}
& \int_0^{t_0}\!\!\int_{\R^d} \partial_t p_{\varepsilon} \ln h_{\varepsilon}(t,x) \rmd x \rmd t
= \int_0^{t_0} \!\! \int_{\R^d} [\partial_t(h_{\varepsilon}\ln h_{\varepsilon})-\partial p_{\varepsilon}] \rmd x \rmd t\\
= &\int_{\R^d}[h_{\varepsilon}(t_0, x)\ln h_{\varepsilon}(t_0, x) -h_{\varepsilon}(0, x)\ln h_{\varepsilon}(0, x)] \rmd x \in L^1(\R^d)
\end{aligned}
\end{equation*}
and
\begin{equation*}\label{25}
\begin{aligned}
h_{\varepsilon}(t_0, x)|\ln h_{\varepsilon}(t_0, x)|
\leq h_{\varepsilon}(t_0, x)(C_1+C_2\Lambda) \in L^1(\R^d),
\end{aligned}
\end{equation*}
we can deduce that
\begin{equation*}\label{26}
\begin{aligned}
h_{\varepsilon}(0, x)\ln h_{\varepsilon}(0, x) \in L^1(\R^d).
\end{aligned}
\end{equation*}

We denote $L_{\varepsilon}:= \int_{\R^d}[h_{\varepsilon}(t_0, x)\ln h_{\varepsilon}(t_0, x) -h_{\varepsilon}(0, x)\ln h_{\varepsilon}(0, x)] \rmd x$ and seek a lower bound for $L_{\varepsilon}$. By applying Jensen's inequality to convex function $x\log x$ and utilizing Shannon's entropy power inequality \cite{C,S2}, we obtain
\begin{equation*}\label{left1}
\begin{aligned}
& \int_{\R^d} h_{\varepsilon}(0,x)\ln h_{\varepsilon}(0,x) \rmd x \\
\leq & \int_{\R^d} p_{\varepsilon}(0,x)\ln(2 p_{\varepsilon}(0,x))\rmd x +\int_{\R^d}\varepsilon \max(|x|, 1)^{-d-1}\ln(2\varepsilon\max(|x|, 1)^{-d-1}) \rmd x \\
\leq & \ln 2+\int_{\R^d}p_{\varepsilon}(0,x)\ln p_{\varepsilon}(0,x) \rmd x +\varepsilon \ln 2\int_{\R^d}\max(|x|, 1)^{-d-1}\rmd x \\
\leq & \ln 2+\int_{\R^d}p_0(x)\ln p_0(x) \rmd x +\varepsilon\ln 2 \int_{\R^d}\max(|x|,1)^{-d-1} \rmd x \\
\leq & \ln 2+ \dfrac{b_1}{\gamma_1}\int_{\R^d}\eta(p_0) \rmd x +\varepsilon\ln 2 \int_{\R^d}\max(|x|,1)^{-d-1} \rmd x
:=L_1(\varepsilon).
\end{aligned}
\end{equation*}
Meanwhile, \eqref{11} yields
\begin{equation*}\label{left2}
\begin{aligned}
& \int_{\R^d} h_{\varepsilon}(t_0, x)\ln h_{\varepsilon}(t_0, x)\rmd x
\geq  -(d+1)\int_{\R^d} h_{\varepsilon}(t_0, x)\Lambda(x)\rmd x\\
\geq & -(d+1)\int_{\R^d} p(t_0, x)\Lambda(x)\rmd x -\varepsilon M_1 (d+1) :=-L_2(\varepsilon).
\end{aligned}
\end{equation*}
Consequently,
\begin{equation}\label{left}
L_{\varepsilon} \geq -L_1(\varepsilon)-L_2(\varepsilon),
\end{equation}
where $L_1(\varepsilon)$ and $L_2(\varepsilon)$ are uniformly bounded with respect to $\varepsilon$.

Now, we prove the square integrability of $\frac{\nabla h_{\varepsilon}}{h_{\varepsilon}}$. By the definition of convolution,
\begin{equation}\label{27}
\partial_i(f(p)\ast \omega_{\varepsilon})=\dfrac{f(p)}{p}\partial_i p_{\varepsilon} +\int_{\R^d}\partial_i\omega_{\varepsilon}(x-y)\Big( \dfrac{f(p(t,y))}{p(t,y)} -\dfrac{f(p(t,x))}{p(t,x)} \Big) p(t, y)\rmd y.
\end{equation}
Elementary computations based on \eqref{19} and \eqref{27} demonstrate that
\begin{equation}\label{28}
\begin{aligned}
& \int_0^{t_0}\!\! \int_{\R^d} \dfrac{f(p)}{p}\dfrac{\partial_i h_{\varepsilon}}{h_{\varepsilon}}\partial_i h_{\varepsilon} \rmd x \rmd t
= \int_0^{t_0} \!\! \int_{\R^d} \dfrac{f(p)}{p}\dfrac{\partial_i h_{\varepsilon}}{h_{\varepsilon}} \Big(\partial_i p_{\varepsilon} +\partial_i\big(\dfrac{\varepsilon}{\max(|x|,1)^{-d-1}}\big) \Big) \rmd x \rmd t \\
= &\int_0^{t_0} \!\! \int_{\R^d} \dfrac{f(p)}{p}\dfrac{\partial_i h_{\varepsilon}}{h_{\varepsilon}} \partial_i\big(\dfrac{\varepsilon}{\max(|x|,1)^{-d-1}}\big) \rmd x\rmd t + \int_0^{t_0} \!\! \int_{\R^d} \dfrac{\partial_i h_{\varepsilon}}{h_{\varepsilon}}\partial_i(f(p)\ast \omega_{\varepsilon})\rmd x\rmd t \\
& -\int_0^{t_0} \!\! \int_{\R^d}\Big\{\dfrac{\partial_i h_{\varepsilon}}{h_{\varepsilon}} \int_{\R^d}\partial_i\omega_{\varepsilon}(x-y) \Big[\dfrac{f(p(t,y))}{p(t,y)} -\dfrac{f(p(t,x))}{p(t,x)} \Big] p(t, y)\rmd y  \Big\} \rmd x\rmd t \\
= &\int_0^{t_0} \!\! \int_{\R^d} \dfrac{f(p)}{p}\dfrac{\partial_i h_{\varepsilon}}{h_{\varepsilon}} \partial_i\big(\dfrac{\varepsilon}{\max(|x|,1)^{-d-1}}\big) \rmd x\rmd t -\int_0^{t_0} \!\! \int_{\R^d} \Big[(\nabla \Phi^i b(p)p)\ast\omega_{\varepsilon}(t,x) \Big]\dfrac{\partial_i h_{\varepsilon}}{h_{\varepsilon}} \rmd x\rmd t \\
& -L_{\varepsilon} -\int_0^{t_0} \!\! \int_{\R^d}\Big\{\dfrac{\partial_i h_{\varepsilon}}{h_{\varepsilon}} \int_{\R^d}\partial_i\omega_{\varepsilon}(x-y) \Big[\dfrac{f(p(t,y))}{p(t,y)} -\dfrac{f(p(t,x))}{p(t,x)} \Big] p(t, y)\rmd y  \Big\} \rmd x\rmd t.
\end{aligned}
\end{equation}
Applying Assumption \ref{con}, we obtain
\begin{equation}\label{29}
\begin{aligned}
& \int_{\R^d}\partial_i\omega_{\varepsilon}(x-y)\Big[\dfrac{f(p(t,y))}{p(t,y)} -\dfrac{f(p(t,x))}{p(t,x)}\Big]p(t,y)\rmd y \\
\leq & \int_{\R^d}(\gamma_2-\gamma_1)\partial_i\omega_{\varepsilon}(x-y)p(t, y) \rmd y
= (\gamma_2-\gamma_1)(p\ast\partial_i\omega_{\varepsilon})(t, x).
\end{aligned}
\end{equation}
Applying H\"older's inequality and \eqref{inequality}, we obtain the estimate
\begin{equation}\label{A1}
\begin{aligned}
& \int_0^{t_0} \!\! \int_{\R^d} \Big[(\nabla \Phi^i b(p)p)\ast\omega_{\varepsilon}(t,x) \Big]\dfrac{\partial_i h_{\varepsilon}}{h_{\varepsilon}} \rmd x\rmd t \\
\leq & \Big(\int_0^{t_0} \!\! \int_{\R^d}\dfrac{|\nabla h_{\varepsilon}|^2}{h_{\varepsilon}} \rmd x\rmd t \Big)^{\frac{1}{2}}
\Big(\int_0^{t_0} \!\! \int_{\R^d} \dfrac{\sum_{i=1}^{d}[(\nabla\Phi^i b(p)p)\ast \omega_{\varepsilon}]^2}{p\ast \omega_{\varepsilon}} \rmd x\rmd t \Big)^{\frac{1}{2}} \\
\leq & \Big(\int_0^{t_0} \!\! \int_{\R^d}\dfrac{|\nabla h_{\varepsilon}|^2}{h_{\varepsilon}} \rmd x\rmd t \Big)^{\frac{1}{2}} \cdot A_1,
\end{aligned}
\end{equation}
where $A_1:=\big(\int_0^{t_0}\int_{\R^d}|\nabla\Phi^i b(p)|^2 p \rmd x \big)^{\frac{1}{2}}$ is a nonnegative constant.
From \eqref{29} and \eqref{inequality}, we have
\begin{equation}\label{A2}
\begin{aligned}
& -\int_0^{t_0} \!\! \int_{\R^d}\Big\{\dfrac{\partial_i h_{\varepsilon}}{h_{\varepsilon}} \int_{\R^d}\partial_i\omega_{\varepsilon}(x-y) \Big[\dfrac{f(p(t,y))}{p(t,y)} -\dfrac{f(p(t,x))}{p(t,x)} \Big] p(t, y)\rmd y  \Big\} \rmd x\rmd t \\
\leq & \Big(\int_0^{t_0} \!\! \int_{\R^d}\dfrac{|\nabla h_{\varepsilon}|^2}{h_{\varepsilon}} \rmd x\rmd t \Big)^{\frac{1}{2}} (\gamma_2-\gamma_1)
\Big(\int_0^{t_0} \!\! \int_{\R^d}\dfrac{(p\ast \partial_i\omega_{\varepsilon})^2}{p_{\varepsilon}} \rmd x\rmd t\Big)^{\frac{1}{2}}\\
\leq & \Big(\int_0^{t_0} \!\! \int_{\R^d}\dfrac{|\nabla h_{\varepsilon}|^2}{h_{\varepsilon}} \rmd x\rmd t \Big)^{\frac{1}{2}} \cdot A_2,
\end{aligned}
\end{equation}
where $A_2:= (\gamma_2-\gamma_1)\big[\int_0^{t_0} \int_{\R^d}\frac{d|\nabla\omega_{\varepsilon}|^2}{\omega_{\varepsilon}}\rmd x \rmd t \big]^{\frac{1}{2}}$ is a nonnegative constant.
Similarly, we can obtain the estimate
\begin{equation}\label{A3}
\begin{aligned}
 \int_0^{t_0} \!\! \int_{\R^d} \dfrac{f(p)}{p}\dfrac{\partial_i h_{\varepsilon}}{h_{\varepsilon}} \partial_i\big(\dfrac{\varepsilon}{\max(|x|,1)^{-d-1}}\big) \rmd x\rmd t
 \leq \Big(\int_0^{t_0} \!\! \int_{\R^d}\dfrac{|\nabla h_{\varepsilon}|^2}{h_{\varepsilon}} \rmd x\rmd t \Big)^{\frac{1}{2}} \cdot A_3,
\end{aligned}
\end{equation}
where $A_3:=\sqrt{\varepsilon} \gamma_2 \big(\int_0^{t_0}\int_{\{|x|\geq 1\}} (d+1)^2|x|^{-d-3} \rmd x \rmd t \big)^{\frac{1}{2}}$ is a nonnegative constant.
Substituting \eqref{A1}, \eqref{A2}, \eqref{A3}, \eqref{left} into \eqref{28}, we get
\begin{equation*}
\begin{aligned}
& \gamma_1 \int_0^{t_0} \!\! \int_{\R^d} \dfrac{\partial_i h_{\varepsilon}}{h_{\varepsilon}}\partial_i h_{\varepsilon} \rmd x \rmd t
 \leq \int_0^{t_0} \!\! \int_{\R^d} \dfrac{f(p)}{p}\dfrac{\partial_i h_{\varepsilon}}{h_{\varepsilon}}\partial_i h_{\varepsilon} \rmd x \rmd t \\
\leq & \Big(\int_0^{t_0} \!\! \int_{\R^d}\dfrac{|\nabla h_{\varepsilon}|^2}{h_{\varepsilon}} \rmd x\rmd t \Big)^{\frac{1}{2}} \cdot (A_1 +A_2 +A_3) +L_1(\varepsilon) +L_2(\varepsilon) \\
= & 2\Big(\dfrac{\gamma_1}{2}\int_0^{t_0} \!\! \int_{\R^d}\dfrac{|\nabla h_{\varepsilon}|^2}{h_{\varepsilon}}\rmd x\rmd t \Big)^{\frac{1}{2}} \cdot
\dfrac{A_1+A_2+A_3}{\sqrt{2\gamma_1}} +L_1(\varepsilon) +L_2(\varepsilon) \\
\leq & \dfrac{\gamma_1}{2}\int_0^{t_0} \!\! \int_{\R^d}\dfrac{|\nabla h_{\varepsilon}|^2}{h_{\varepsilon}}\rmd x\rmd t +\dfrac{(A_1+A_2+A_3)^2}{2\gamma_1}
+L_1(\varepsilon) +L_2(\varepsilon).
\end{aligned}
\end{equation*}
Therefore,
\begin{equation*}
\begin{aligned}
 \int_0^{t_0} \!\! \int_{\R^d} \dfrac{|\nabla h_{\varepsilon}|^2}{h_{\varepsilon}} \rmd x\rmd t
 \leq \dfrac{(A_1+A_2+A_3)^2}{\gamma_1^2} +\dfrac{2}{\gamma_1}L_1(\varepsilon) +\dfrac{2}{\gamma_1}L_2(\varepsilon).
\end{aligned}
\end{equation*}
Let $\varepsilon \rightarrow 0$, we obtain
\begin{equation*}\label{M}
\begin{aligned}
& \int_0^{t_0} \!\! \int_{\R^d} \dfrac{|\nabla p(t,x)|^2}{p(t,x)} \rmd x\rmd t \\
\leq & \dfrac{(A_1+A_2)^2}{\gamma_1^2} +\dfrac{2}{\gamma_1}\Big(\ln 2+\dfrac{b_1}{\gamma_1}\int_{\R^d}\eta(p_0) \rmd x +(d+1)\int_{\R^d} p(t_0, x)\Lambda(x)\rmd x\Big)< +\infty.
\end{aligned}
\end{equation*}
This completes the proof.

\section*{Acknowledgements}
This work is supported by National Key R\&D Program of China (No. 2023YFA1009200), NSFC Grant 11925102, and Liaoning Revitalization Talents Program (Grant XLYC2202042).


\end{document}